\documentclass[12pt,journal,onecolumn]{IEEEtran}

\usepackage{setspace}
\usepackage{amssymb,amsmath,amsthm,mathrsfs,mathtools,bm,bbm}
\usepackage[dvipsnames]{xcolor}
\usepackage{soul}
\usepackage{times}
\usepackage{graphicx}
\usepackage{dblfloatfix}
\usepackage{caption,nicefrac}
\usepackage{subcaption}
\usepackage[shortlabels]{enumitem}
\usepackage{makecell}

\usepackage{algpseudocode}
\usepackage{algorithm}
\usepackage{cite}

\makeatletter
\newcommand{\markthis}[3]{
  \overset{
    \textup{\makebox[0pt]{#1}}%
    \def\@currentlabel{#1}%
    \ltx@label{#2}%
  }{
    #3%
  }%
}
\makeatother

\newtheorem*{proposition*}{Proposition}
\newtheorem*{theorem*}{Theorem}
\newtheorem{assumption}{Assumption}
\newtheorem{lemma}{Lemma}
\newtheorem{theorem}{Theorem}
\newtheorem{definition}{Definition}
\newtheorem{problem}{Problem}

\newtheorem{example}{Example}

\newtheorem{proposition}{Proposition}

\usepackage{tikz}
\usetikzlibrary{arrows}
\usetikzlibrary{automata,arrows,positioning,calc}
\usepackage{circuitikz}
\usetikzlibrary{shapes}

\definecolor{babypink}{rgb}{0.96, 0.76, 0.76}
\definecolor{bananayellow}{rgb}{1.0, 0.88, 0.21}
\definecolor{amethyst}{rgb}{0.6, 0.4, 0.8}
\definecolor{blizzardblue}{rgb}{0.67, 0.9, 0.93}
\definecolor{aquamarine}{rgb}{0.5, 1.0, 0.83}
\definecolor{aureolin}{rgb}{0.99, 0.93, 0.0}
\definecolor{aqua}{rgb}{0.0, 1.0, 1.0}
\definecolor{caribbeangreen}{rgb}{0.0, 0.8, 0.6}
\definecolor{chartreuse(web)}{rgb}{0.5, 1.0, 0.0}
\definecolor{amber(sae/ece)}{rgb}{1.0, 0.49, 0.0}
\definecolor{apricot}{rgb}{0.98, 0.81, 0.69}
\definecolor{lightmauve}{rgb}{0.86, 0.82, 1.0}
\definecolor{lightsalmon}{rgb}{1.0, 0.63, 0.48}
\definecolor{electricblue}{rgb}{0.49, 0.98, 1.0}
\definecolor{gray(x11gray)}{rgb}{0.75, 0.75, 0.75}

\newcommand{\EE}{\mathbb{E}}

\newcommand{\calI}{\mathcal{I}}

\DeclareMathOperator*{\argmin}{argmin}

\definecolor{green1}{rgb}{0.2,0.7,0.2}

\title{\LARGE \bf Semantic Communication in Multi-team Dynamic Games: A Mean Field Perspective}

\author{Shubham~Aggarwal, Muhammad~Aneeq~uz~Zaman, Melih Bastopcu, \textit{Member, IEEE}, and Tamer~Ba{\c s}ar, \textit{Life Fellow, IEEE}\vspace{-0.5cm}
\thanks{Research of the authors was supported in part by ARO MURI Grant AG285 and in part by AFOSR Grant FA9550-24-1-0152. Research of MB was supported in part by Tubitak Bilgem EDGE-4-IoT and Tubitak 2232-B Fellowship Program under Project No. 124C533.}
\thanks{Shubham Aggarwal and Muhammad Aneeq uz Zaman are with the Coordinated Science Laboratory and the Department of Mechanical Science and Engineering at the University of Illinois Urbana-Champaign (UIUC); Melih Bastopcu is with the Department of Electrical and Electronics Engineering at Bilkent University, Türkiye;  Tamer Ba{\c s}ar is with the Coordinated Science Laboratory and the Department of Electrical and Computer Engineering at UIUC. (Emails:
        {\{sa57, mazaman2, bastopcu, basar1\}@illinois.edu)} 
}
}
\tikzset{ remember picture,
   switch/.style = {rectangle,
                    draw,align=center,
                    label={below:#1},
   },
}

\newsavebox\mybox
\savebox\mybox{%
\tikz\draw[line width=0.7pt] (-0.4,0)--(0,0)
								(0,0)--(0.4,0.4);%
}
\begin{document}
\tikzstyle{rect} = [draw,rectangle,fill = white!20,minimum width = 3pt, inner sep  = 5pt]
\tikzstyle{line} = [draw, -latex]
\tikzstyle{dline} = [draw, dash dot, -latex]

\maketitle
\thispagestyle{empty}

\begin{abstract}
Coordinating communication and control is a key component in the stability and performance of networked multi-agent systems. 
While single user networked control systems have gained a lot of attention within this domain, in this work, we address the more challenging problem of large population multi-team dynamic games.
In particular, each team constitutes two decision makers (namely, the sensor and the controller) who coordinate over a shared network to control a dynamically evolving state of interest under costs on both actuation and sensing/communication. 
Due to the shared nature of the wireless channel, the overall cost of each team depends on other teams' policies, thereby leading to a noncooperative game setup. Due to the presence of a large number of teams, we compute approximate decentralized Nash equilibrium policies for each team using the paradigm of (extended) mean-field games, which is governed by (1) the mean traffic flowing over the channel, and (2) the value of information at the sensor, which highlights the semantic nature of the ensuing communication. In the process, we compute optimal controller policies and approximately optimal sensor policies for each representative team of the mean-field system to alleviate the problem of general non-contractivity of the mean-field fixed point operator associated with the finite cardinality of the sensor action space.
Consequently, we also prove the $\epsilon$--Nash property of the mean-field equilibrium solution which essentially characterizes how well the solution derived using mean-field analysis performs on the finite-team system. 
Finally, we provide extensive numerical simulations, which corroborate the theoretical findings and lead to additional insights on the properties of the results presented. 
\end{abstract}

\section{Introduction}
Games among teams have been an attractive topic of interest in the recent past. These systems consist of cooperating decision-makers (within a team) who aim to collectively optimize the shared team objective, while being affected by the actions of other competing teams. Applications of the same include multi-player perimeter defense games \cite{shishika2020review}, electricity markets \cite{subramanian2023mean}, oligopolies \cite{weintraub2008markov}, and networked robotics \cite{hatanaka2015passivity}, to list just a few.

In this paper, we address the problem of finding equilibrium policies for multiple dynamic teams competing to access a shared wireless channel. A prototype of such a system is shown in Fig.~\ref{Fig:system_model}.  Each team is comprised of a dynamically evolving plant controlled by a sensor-controller pair. The controller relies on state information from the sensor, which communicates over the shared wireless channel. Costs arise from (1) the sensor transmitting data over the network, and (2) the controller implementing control policies for plant stabilization. An effective control policy requires consistent state measurements, which increase communication costs. Conversely, reducing the number of communications lowers these costs, but degrades control performance, leading to a classic communication-control trade-off for each team. Consequently, the sensor and controller act as two active decision-makers within a team, optimizing a common objective to find a team-optimal policy pair.

\begin{figure}[t]
    \centerline{\includegraphics[width=0.6\columnwidth]{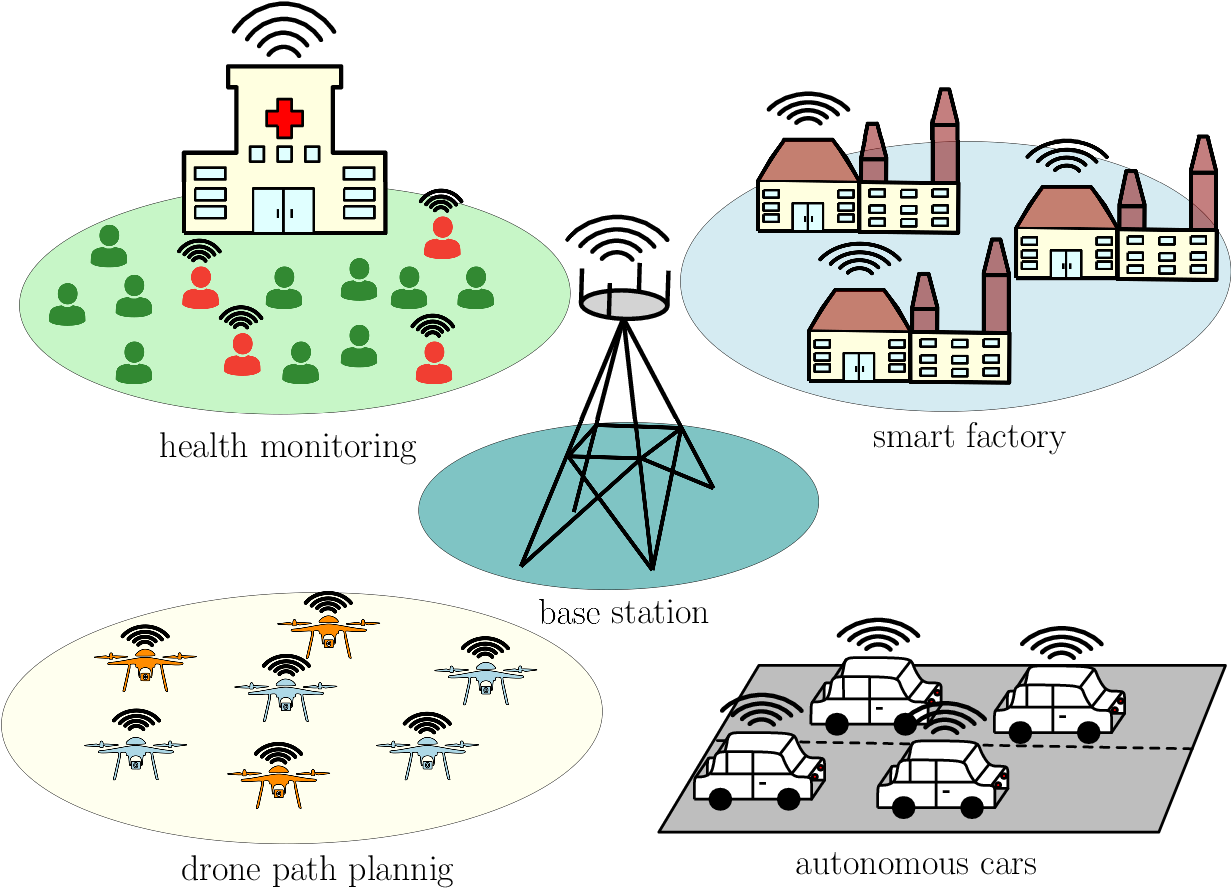}}
    \caption{\small{A multi-agent networked control system operating over a shared wireless resource serves use cases such as connected autonomy, path planning, Industry 4.0, and medical applications.}}
    \label{Fig:system_model}
    \vspace{-0.6cm}
\end{figure}
Furthermore, since the channel is shared, the communication cost is coupled with the transmission activity of other users. We model this coupling effect as the product of sensor action and the average number of communication instances of the other teams. Therefore, if the channel is highly congested, each sensor will seek to reduce its communication, whereas if the channel is underutilized, each sensor will increase its transmissions to improve control performance. This dynamic interaction suggests that a (Nash) equilibrium channel utilization may emerge for each team. However, with a large number of teams, calculating exact Nash policies is challenging because each team needs to be aware of other teams' policies. To address this, we employ the mean-field game (MFG) framework to compute approximate \textit{decentralized} equilibrium policies for each team, which depend only on the local information. This approach involves analyzing the limiting team game problem (with infinite teams), which we refer to as the mean-field team (MFT) system. In the MFT system, individual team deviations become insignificant due to the infinite number of teams, and hence, it can be characterized by a representative team interacting with a mass distribution. This allows us to determine the optimal policy for a generic team, referred to as the mean-field team equilibrium (MFTE) policy, which aligns with the population's behavior. We provide conditions for the existence and uniqueness of an approximate MFTE and demonstrate when the derived equilibrium policies constitute approximate Nash policies for the finite-team system.
\vspace{-0.2cm}
\subsection{Related Works}

\subsubsection{Team optimal control \& Information structures}\label{subsec:Team_optimal}
Team decision problems, which were first investigated in \cite{marschak1955elements,radner1962team}, involve multiple decision makers (DMs) each of whom has access to different information variables and consequently choose policies jointly to incur a common cost/reward. Since each one acts independently, and they do not necessarily share the same information, the joint optimal policy design is highly dependent on the information available to each DM \cite{yuksel2013stochastic,dave2019decentralized,dave2022decentralized}, and its derivation can be challenging particularly when the information is dynamic, as demonstrated by Witsenhausen \cite{witsenhausen1968counterexample}, Feldbaum \cite{feldbaum1961dual}, and Ba\c{s}ar \cite{basar2008variations}. The \textit{single team} communication-control trade-off problem is an instance of such a decision problem, where the sensor and the controller form a 2-DM team, each having access to different information signals within the system. Initial investigations for solving the single team problem in this context were undertaken within the realm of event-driven control, where the primary objective is to merely stabilize the dynamical plant, without any optimality guarantees \cite{tabuada2007event,heemels2012introduction}. Multi-team extensions of the same have also been considered in \cite{dimarogonas2011distributed,seyboth2013event}.

On the other hand, single team settings from the viewpoint of an optimal control problem have also been widely dealt with in literature \cite{imer2005optimal,imer2010optimal,lipsa2011remote,imer2006optimal,imer2006measure,molin2009lqg,maity2020minimal} to study the impact of resource constraints on estimator and controller designs, where the resource is the frequency of usage of the transmission channels for estimation or control purposes. Specifically, papers \cite{imer2005optimal} and \cite{imer2010optimal} have obtained the optimal threshold-type (symmetric) policies for the sensor under a hard constraint on the number of transmissions, with extensions provided in \cite{lipsa2011remote}; 
\cite{imer2006optimal} has placed resource constraints on the interaction between the controller and the actuator in a plant; and \cite{imer2006measure} has
introduced a tradeoff between two options for a controller, which are transmission of control signals to the actuator and receiving measurements, where optimality is again based on threshold-type policies; \cite{molin2009lqg} has provided jointly optimal policies for control and sensing; and \cite{maity2020minimal} has derived suboptimal sensor and control policies based (only) on the statistics of the estimation error.
The latest works in this direction are \cite{soleymani2021valu,soleymani2022value}, where the authors show that for a multi-dimensional Gauss-Markov process, the optimal estimator is linear and is unaffected by the \textit{no-communication events} (as was derived for \textit{scalar} systems earlier in \cite{imer2010optimal} with its optimality proven in \cite{lipsa2011remote}), leading to a globally optimal policy design for a single team problem.

\subsubsection{Finite \& Mean-field team games}
As opposed to the above frameworks which consider a single team problem, there are only a handful of works considering games among multiple teams\cite{hogeboom2023zero,lagoudakis2002learning,ghimire2023solving,maity2017linear}. Specifically, the paper \cite{hogeboom2023zero} has investigated the properties of the saddle-point equilibrium in a two-team zero-sum game;  \cite{lagoudakis2002learning} has proposed a learning algorithm for zero-sum Markov games  between teams of multiple agents, with extensions to general sum 2-team games in \cite{ghimire2023solving}; and \cite{maity2017linear} has provided only the Nash controller policies for 2-team general sum games without explicit sensor policy computation, which becomes challenging even for a 2-team problem.

To alleviate the above issue of equilibrium policy computation, the framework of mean-field games (MFGs) (which forms a subset of the mean-field team games with each team constituting a single DM) was introduced simultaneously in \cite{huang2007large} and \cite{lasry2007mean}. The idea behind the same is to compute approximate decentralized Nash equilibrium policies for agents by passing to an infinite agent system (under suitable assumptions of anonymity between agents), in which case the effect of actions of any specific agent becomes insignificant on that of the other agents, thereby leading to a decoupling effect.
Under this scenario, one can solve for a representative agent's equilibrium policy within the infinite agent system, which is consistent with that of the population and then characterize its approximation with respect to the centralized Nash equilibrium policies for the finite-agent system.

While many works have considered both theoretical and application aspects of standard MFGs \cite{huang2006large,aggarwal2023weighted,bagagiolo2014mean,olmez2022modeling,aggarwal2024mean}, only a few results exist pertaining to MFTGs \cite{huang2024linear,subramanian2023mean}. A similar paradigm also worth mentioning is that of \emph{mean-field type games} \cite{djehiche2016mean,zaman2024independent} where the number of teams is finite, but the number of agents within each team goes to infinity. An added challenge within the MFTG setup is to solve for the team optimal solution for each team whilst handling the incoming interactions from other teams.
While the work \cite{subramanian2023mean} considers a discrete state and action setup, the work \cite{huang2024linear} considers a linear-quadratic setting to compute an asymptotic mixed-equilibrium optima for the team-game problem. In our work, however, we consider a mixed \textit{discrete-continuous} action space where the controller action lies within the $m$--dimensional Euclidean space while the sensor action is binary-valued (either 0 denoting no transmission or 1 denoting full state transmission). Further, as we will see, the cost function is non-quadratic and the estimation error dynamics, which forms the state of the estimated problem, is bilinear due to the presence of no-transmission events, which makes the problem significantly more challenging.

\subsubsection{Semantic communication}
With the proliferation of the Internet-of-Things (IoT) technology, there has been an increasing interest in the \textit{efficient} use of the shared wireless spectrum \cite{aggarwal2023weighted,aggarwal2023large,mason2024multi}. To enable the same, semantic communication is an emerging paradigm that aims to transmit only \textit{relevant} information through shared resources rather than just the raw data, leading to more intelligent and context-aware communication systems, especially in bandwidth-constrained environments \cite{lan2021semantic}.
For instance, in \cite{uysal2022semantic}, it has been argued that with semantic communication, it is possible to reduce the transmission rate without significant loss in the system performance; the paper \cite{du2023rethinking} discusses the potential energy savings resulting from semantic interaction in connected autonomous vehicles; and the paper \cite{kountouris2021semantics} demonstrates the effectiveness of semantics powered sampling for remote real-time source reconstruction. We refer the interested reader to \cite{ayan2022semantics,yang2022semantic}, and the references therein for additional details. In this work, we capture the semantics of information using the value of information (VoI) metric \cite{soleymani2021valu}. This metric essentially measures the (excess) payoff derived by a team from the information communicated by the sensor to the controller as opposed to no communication, thereby providing a quantitative basis for evaluating the effectiveness of semantic communication in our proposed approach.

Thus, we list the main contributions of our work as follows.

\begin{enumerate}
    \item We address the problem of equilibrium computation and analysis in shared-channel multi-team noncooperative dynamic games. The complexity introduced by the presence of a large number of teams renders exact computations intractable. To overcome this, we leverage the framework of mean-field team games to compute approximate decentralized Nash equilibrium (controller-sensor) policies for each team (as established in Proposition \ref{prop:existence} and Theorem \ref{thm:uniqueness}). The novelty of our approach lies in addressing the general non-contractivity of the mean-field operator, caused by the finite cardinality of the sensor action space. This non-contractivity precludes the existence of a mean-field team equilibrium (MFTE). To resolve this challenge, we introduce probabilistic Boltzmann-type policies to approximate the optimal sensor policies which effectively smoothen out the non-contractivity issues by introducing stochasticity into the decision-making process as a function of the average utilization of the wireless channel (or the mean traffic flow across the network). This formulation subsequently allows us to establish the existence and uniqueness of an approximate MFTE.

    \item Our second contribution is to establish the reasonability and practicality of applying the Mean-Field Team Equilibrium (MFTE) solution to finite multi-team games (as formally introduced in Definition \ref{defn:epsNash}). To address this, we theoretically demonstrate (in Theorem \ref{thm:eps_Nash}) that the MFTE serves as an approximate Nash equilibrium for the finite-team system. 
    The approximation arises from two main factors. First, to derive the MFTE solution, we rely on a $\alpha$--parameterized Boltzmann policy, which introduces a level of approximation due to its probabilistic nature. In this case, we show (partially in Proposition \ref{prop:Second_term_bound}) that the degree of approximation decays exponentially as a function of $\alpha$.
    Second, the MFTE solution is based on the infinite-team mean-field framework, which starts with an infinitely large number of teams and uses the mean-field approximation to model their interactions. When this solution is applied to a system with a finite number of teams (denoted by $N$), the inherent discrepancy between the infinite-team model and the finite-team system introduces an additional source of approximation (which we show decays with order $1/\sqrt{N}$ and an additional asymptotic term in $N$, using helper Lemmas \ref{lem:Approx_MFE}, \ref{eq:First_term_bound}, \ref{eq:Third_term_bound} and \ref{lem:Q_bounded}).
\end{enumerate}

We would also like to highlight that our current work employs optimal policy notions introduced in \cite{soleymani2021valu}, which focused on a single-team optimal control problem, and constitutes a special instance of our work where the number of teams, $N$, is equal to 1. However, transitioning from a single-team optimal control problem to a multi-team noncooperative dynamic game introduces significant new challenges.
The key non-triviality in our work lies in solving for Nash equilibrium policies in addition to the team-optimal policies for each team. Within this setup, we must account for strategic interactions between multiple teams operating in a shared-channel environment, which can especially be difficult within a large population scenario. By leveraging the MFTG framework, we overcome these challenges and provide a structured approach to compute approximate Nash policies (primarily through Theorems \ref{thm:uniqueness} and \ref{thm:eps_Nash}).

\vspace{-0.2cm}
\subsection{Organization of the paper}
The paper is organized as follows. We formulate the finite-team game problem in Section \ref{sec:formulation} and the mean-field team game in Section \ref{sec:MF_team_Game}. The analysis of the existence and uniqueness of the MFTE is presented in Section \ref{sec:eqb_analysis}, its approximation performance is studied in Section \ref{sec:perf_approx}, and the approximate Nash equilibrium analysis is included in Section \ref{sec:eps_nash}. Numerical results are presented in Section \ref{sec:num_exp}, and the paper concludes with some major highlights and ensuing discussions in Section \ref{sec:conclusion}. Proofs of main propositions and the supporting lemmas are presented in the Appendices.

\textbf{Notations:} The set of agents is denoted by $[N]:= \{1, \cdots, N\}$ and the set of time instants by $[T]:= \{0, \cdots, T-1\}$. We denote the Euclidean norm by $\|\cdot\|$. For a positive definite (resp. semi-definite) matrix $Z$, we use the notation $Z \succ 0$ (resp. $Z \succeq 0$). The trace of a matrix $Z$ is denoted by $tr(Z)$, and for a vector $x$, we define $\|x\|^2_Z:= x^\top Z x$. For two real functions $g$ and $h$,  $g(x) = \mathcal{O}(h(x))$ means that there exists a constant $C>0$ such that $\lim_{x \rightarrow \infty} |g(x)/h(x)| = C$. We define $Z_{0:m} := \{Z_0, \cdots, Z_m\}$, and $\mathbb I[A]$ as the indicator function of a measurable set $A$. 
\vspace{-0.2cm}
\section{Finite-Team Game Problem}\label{sec:formulation}\vspace{-0.1cm}
Consider the networked multi-agent system shown in Fig. \ref{Fig:Inf_flow}, comprised of $N$ teams. Each team $i$ is modeled by a triple $(P_i,S_i,C_i)$, i.e., a dynamically evolving plant ($P_i$), a sensor ($S_i$), and a controller ($C_i$), where $S_i$ and $C_i$ can be viewed as the two active decision makers of team $i$. The plant $P_i$ evolves according to a stochastic linear difference equation as
\begin{align}\label{System_dynamics}
    X^i_{k+1} = A(\omega_\ell)X^i_k + B(\omega_\ell)U^i_k + W^i_k, ~~i \in [N],
\end{align}
where $X^i_k \in \mathbb{R}^n$ denotes the state, and $U^i_k \in \mathbb{R}^m$ denotes the control input, both for the $i^{th}$ team. The system noise $W^i_k$ is assumed to be Gaussian
with zero mean and finite covariance $K_{W^i}$. Further, $A(\omega_\ell)$ and $B(\omega_\ell)$ are matrices with appropriate dimensions which depend on the type of the team $\omega_\ell$ with $\omega_\ell \in \Omega$, and are chosen according to empirical probability mass function (pmf) $\mathbb{P}_N(\omega =\omega_\ell)$, for all teams. We take the set $\Omega$ of possible types to have finite cardinality.
 The initial state $X_0^i$ of a team of type $\omega_\ell$ has a symmetric density with mean $\nu(\omega_\ell)$ and positive definite covariance $\Sigma_0(\omega_\ell)$ for each team. We assume that $X_0^i$ is independent of noise $W^i_k$ for all $k,i$, and $\lim_{N \rightarrow \infty} \mathbb P_N(\omega) = \mathbb P(\omega), \forall \omega \in \Omega$.
 
 \begin{figure}[t]
	\centerline{\includegraphics[width=0.6\columnwidth]{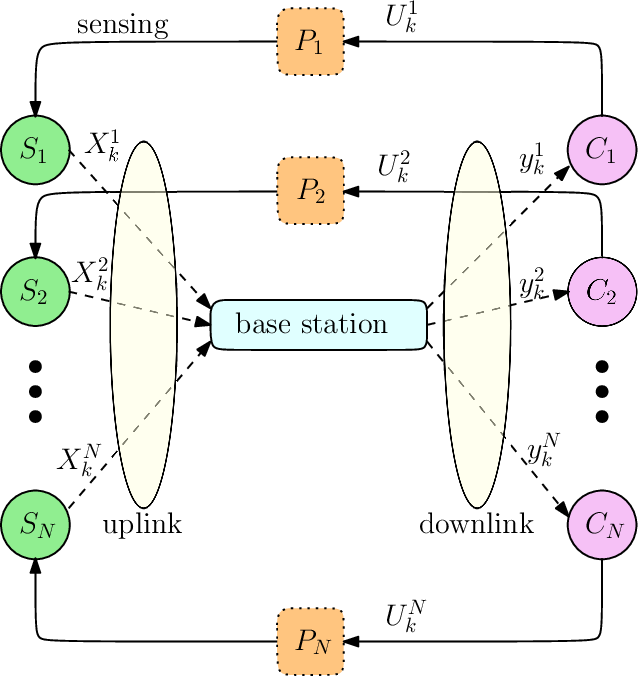}}
	\caption{\small{The figure shows a prototypical networked multi-agent system constituting $N$ teams of plant $P_i$, sensor $S_i$ and controller $C_i$. The \textit{shared} wireless link between the sensor and the respective controller constitutes two wireless hops, namely, the uplink and the downlink. The shared channel couples the sensing policies of each team leading to a noncooperative game setup between teams.}}
	\label{Fig:Inf_flow}
    \vspace{-0.4cm} 
\end{figure}
Next, for each team $i$, full-state information ($X^i_k$) of the plant is relayed by $S_i$ to $C_i$ over a shared two-hop wireless channel via a central base station  which coordinates the flow of information (as illustrated in Fig. \ref{Fig:Inf_flow}). We refer to the hop from the sensor to the BS as the uplink and the one from the BS to the controller as the downlink. The sensory information is relayed to the controller (via the uplink and the downlink) subject to a one-step delay, i.e., if $X^i_{k-1}$ and $y^i_k$ denote the uplink input and downlink output, respectively, then we have that
\begin{align}\label{decoder_signal}
    y^i_{k+1} = \left\{ \begin{array}{ll}
X^i_k, & \text{if $\gamma^i_k = 1$} \\ 
\varphi, & \text{if} ~\gamma^{i}_{k} = 0 \\
\end{array},
\right.
\end{align}
for $k \in [T]$ and $y_0 = \varphi$. The quantity $\varphi$ denotes an empty symbol and $\gamma^i_k$ denotes the transmission requests of $S_i$ as \begin{align}\label{transmission}
   \gamma^{i}_k = \left\{ \begin{array}{ll}
1, & \text{if transmission is requested} \\ 
0, & \text{otherwise} \\
\end{array}.
\right.
\end{align}

We further assume that the downlink has enough bandwidth to accommodate all the sensor requests.\footnote{The case where the downlink suffers from a capacity constraint with an ideal uplink is dealt with in the works \cite{aggarwal2023weighted,aggarwal2023large}, where the authors use a weighted age of information-based metric to devise optimal downlink transmission strategies for the base station. However, the case where both the uplink and the downlink suffer from non-idealities is a significantly challenging problem since one would need to jointly design each team's policies and the scheduling policy of the base station. We will leave this as an open and promising research direction.} 
Whenever a transmission request is initiated, a connection is established between the sensor and the controller via the base station. However, each packet transmitted from the sensor to the controller incurs a running cost for team $i$ which is proportional to the \textit{average number of users communicating on the uplink}. Specifically, for given sensor actions $\gamma^j_k$'s of the sensors $S_j$'s of all teams, this cost for team $i$ is given as:
\begin{align}\label{Scheduling_Cost}
    c_1^i(\gamma^1_k, \cdots, \gamma^N_k) := \gamma^i_k \Bigg( \frac{1}{N}\sum_{j \in N} \gamma^{j}_k \Bigg),
\end{align}
where the term $\frac{1}{N}\sum_{j \in [N]} \gamma^{j}_k$ denotes \textit{average channel utilization} due to all the teams and serves to model the mean traffic flowing through the wireless channel at instant $k$. Consequently, team $i$ incurs a high communication cost if $S_i$ transmits to $C_i$ (i.e., $\gamma^i_k = 1$) when the channel is experiencing a high load. Hence, the product of $\gamma^i_k$ and the average utilization term effectively discourages excessive use of the shared spectrum and this formulation can thus be interpreted as a soft cost penalty modeling of a bandwidth-constrained wireless channel.

Upon reception of the transmitted data packet, the controller $C_i$ utilizes it to determine an appropriate control action, which also incurs an actuation cost given as:
\begin{align}\label{LQT_cost}
    c_2^i(X_k^i,U_k^i) := \|X^i_k\|^2_{Q(\omega_\ell)} \!+\! \|U^i_k\|^2_{R(\omega_\ell)},
\end{align}
where $\omega_\ell$ denotes the type of team $i$, $Q(\omega_\ell) \succeq 0$, $R(\omega_\ell) \succ 0$ are type-dependent matrices, and the cost given in \eqref{LQT_cost} incorporates a soft constraint on the control input alongside penalizing state deviations from the origin.
Let us introduce the information set of the $i^{th}$ team's sensor and controller as
\begin{subequations}\label{eq:info_cent}
    \begin{align}
    \calI^{S,i}_k & := \{X^i_{0:k},y^i_{0:k}, U^i_{0:k-1},\gamma^{\ell}_{0:k-1} \mid \ell \in [N]\} \\
    \calI^{C,i}_k & :=\left\lbrace y^i_{0:k},U^i_{0:k-1} \right\rbrace,
\end{align}
\end{subequations}
for $k \in \{1, 2, \cdots, T-1\}$ with $\calI^{S,i}_0 := \{X^i_0,y^i_0\}$ and $\mathcal{I}^{C,i}_0 := \{y^i_0\}$. Further, let us define the maps $\xi^i_k:\calI^{S,i}_k \to \{0,1\}$ and $\pi^i_k:\calI^{S,i}_k \to \mathbb R^{m}$ with $\xi^i = \{\xi^i_k\}_{\forall k}$ and $\pi^i = \{\pi^i_k\}_{\forall k}$ denoting the $i^{th}$ team's sensor policy and controller policy, respectively, which lie in the space of admissible \textit{centralized}\footnote{Here we use the term \textit{centralized} to emphasize that the information set of sensor $i$ contains the transmission instants of other agents as well.} sensor policies (denoted as $\Xi^{cen,i}$) and controller policies (denoted as $\mathcal{M}^i$), respectively, given by
\begin{align*}
    \Xi^{cen,i} & := \{\xi^i \mid \xi^i_k ~\text{is adapted to } \sigma(\calI^{S,i}_t, ~t = 0,1, \cdots, k)\} \\
    \mathcal{M}^i & := \{\pi^i \mid \pi^i_k ~\text{is adapted to } \sigma(\calI^{C,i}_s, ~s = 0,1, \cdots, k)\},
\end{align*}
where $\sigma(\cdot)$ is the $\sigma$--algebra generated by its arguments. We will refer to $\mu^i:=(\pi^i,\xi^i)$ as the $i^{th}$ team's policy. 
Consequently, the overall average cost incurred by team $i$ over a finite horizon of $T > 0$ is given as:
\begin{align}
    J^N_i(\mu^i, \mu^{-i}) := \frac{1}{T} \mathbb E \left[\sum_{k=0}^{T-1}  c_2^i(X_k^i,U_k^i) + \lambda c_1^i(\gamma^1_k, \cdots, \gamma^N_k) \right],
\end{align}
where we let $\mu^{-i}:= (\mu^1,\cdots, \mu^{i-1}, \mu^{i+1}, \cdots, \mu^N)$ to be the vector of policies of teams other than the $i^{th}$ one. Further, $\lambda >0$ is the weighting parameter which trades off the costs between uplink communication and control performance.
Finally, the expectation above is taken with respect to the stochasticity induced by the (possibly) randomized policies, the system noise, and the initial state distribution. As a result, each team is faced with the following stochastic team-game problem.

\begin{problem}[Team-level game problem]\label{Prob:Stoch_min_Overall}
Each team $i$ wishes to solve the following optimization problem:
\begin{align*}
    & \inf_{\mu^i\in \mathcal{M}^i \times \Xi^{cen,i}} \! J^N_i(\mu^i,\mu^{-i}) \\
    & \mbox{s.t.} ~\eqref{System_dynamics},~\eqref{decoder_signal}\text{, and }  \eqref{eq:info_cent} \text{ hold}.
\end{align*}
\end{problem}
In the above problem, we first observe that for each team, a higher communication rate between the sensor and the controller enhances control performance due to more frequent state updates; however, this comes at the expense of increased communication costs. Conversely, a lower communication rate reduces communication overhead but degrades control performance. Hence, the sensor and the controller must \textit{judiciously} select their respective policies to balance these competing objectives and achieve optimal system performance. Moreover, each team's decision-making process is influenced by the \textit{strategies} of the other teams through their respective sensor policies (and the shared nature of the wireless medium), leading to a noncooperative game-theoretic setup.
Notably, prior works \cite{soleymani2021valu,soleymani2022value} addressed a special case of our framework, where the problem is limited to a single team ($N=1$). In contrast, our objective in this work is to compute Nash equilibrium policies for the $N$ teams. That is, we are looking for a $2N$--tuple $(\mu^{*,i} \in {\cal M}^i \times \Xi^{cen,i}, i \in [N])$ such that
\begin{align*}
J^N_i(\mu^{*,i}\mu^{*,-i}) = \inf_{\mu^i \in {\cal M}^i \times \Xi^{cen,i}} J^N_i(\mu^{i}, \mu^{*,-i})
\end{align*}
for all $i \in [N]$. The above set of equations state that if teams other than the $i^{th}$ one follow the Nash equilibrium policy, then it is also in the best interest of team $i$ to follow the equilibrium policy $\mu^{*,i}$. However, the primary challenge in computing these equilibrium policies arises from the presence of a large number of teams in typical networked systems. In such large-scale environments, it is neither feasible nor practical for each team to obtain and store the complete strategy information of all other teams. Given this limitation, our goal in the rest of this paper is to derive \textit{(approximate decentralized) team (Nash) equilibrium policies} that rely solely on the local information accessible to each team, and the statistics of the underlying game. To achieve this, in the next three sections, we will leverage the mean-field game framework \cite{huang2006large,lasry2006jeux,lasry2006jeux1}, which provides a principled approach for computing the same. In passing, we also note here that a key real-world application of the above formulation is OptiTrack motion capture systems \cite{li2021real}, which serve as prototypes for both indoor navigation systems and outdoor navigation in GPS-denied/noisy-GPS environments. We will simulate such a system later in the numerical section to validate our proposed approach, which is presented next.

\section{Networked Mean-Field Team Game}\label{sec:MF_team_Game}

In this section, we formulate the mean-field team game (MFTG) with infinite teams corresponding to the $N$--team game of the previous section. Under this asymptotic setting, the average channel utilization ratio in \eqref{Scheduling_Cost} (and consequently in Problem \ref{Prob:Stoch_min_Overall}) can be approximated by deterministic \textit{mean-field trajectory} by using the Nash certainty equivalence principle \cite{huang2007large}. The rationale behind the same is that the aggregation of sensor policies across all teams, as captured in \eqref{Scheduling_Cost}, leads to a deterministic effect in the asymptotic limit as the number of teams approaches infinity. This phenomenon induces a \textit{decoupling effect}, whereby the originally coupled multi-team optimization problem (Problem \ref{Prob:Stoch_min_Overall}) simplifies to a team-optimal control problem for a single representative team, the solution of which is consistent with the solutions of other teams in the population. However, it is important to emphasize that this reasoning is, at this stage, a heuristic argument. A rigorous justification requires a precise mathematical characterization of the interaction between individual teams and the collective population. Establishing this formal connection and proving the validity of the mean-field approximation will be the central focus of the analysis that follows.

To proceed with the same, we first provide the team-optimal policy for a representative team under the mean-field approximation. However, a key challenge arises due to the finite cardinality of the sensing action set (given as $\{0,1\}$), which leads to the \textit{non-contractivity} of the mean-field operator (as previously also noted in \cite{cui2021approximately}), complicating the direct application of standard fixed-point techniques. To address this issue, we introduce a Boltzmann-type sensor policy that smooths the sensor policy by incorporating stochasticity. Using this formulation, we subsequently establish the existence of a unique \textit{approximate mean-field team equilibrium} (MFTE). Finally, we rigorously show that the derived MFTE serves as an \textit{approximate Nash equilibrium} for the finite $N$--team game, with ``approximation" defined in a precise way. A schematic flow of the solution approach is also later presented in Fig. \ref{Fig:flowchart}.

\subsection{Mean-Field Team Game}\label{subsec:MFTG}
Consider a representative team (within the infinite team system) of type $\omega$, characterized by the tuple $(P(\omega),S(\omega), C(\omega))$ of a plant, a sensor, and a controller. The plant dynamics evolve according to the stochastic linear difference equation
\begin{align}\label{systemMF}
    X_{k+1} = A(\omega) X_k + B(\omega)U_k + W_k,
\end{align}
where $(X_k,U_k) \in \mathbb{R}^n \times \mathbb{R}^m$ denotes the state-control input pair of the representative team. Noise $W_k$ is assumed to be Gaussian with zero mean and positive definite covariance $K_W(\omega)$. The sensor $S(\omega)$ generates transmission requests ($\gamma_k \in \{0,1\}$) over the uplink, which is then coordinated by the base station to update the state of the plant to the controller. As a result of network usage, the team incurs a running cost given as
\begin{align*}
    \\[-45pt]
\end{align*}
\begin{align}\label{Scheduling_CostMF}
    c_1^\omega(\gamma_k, g_k) := \gamma_k g_k,
\end{align}
where $\gamma_k$ is a sensor action chosen using the mapping $\zeta_k:\calI^{S,\omega}_k \rightarrow \{0, 1\} $ with $\zeta := \{\zeta_k\}_{\forall k} \in \Xi^{dec}:= \{\zeta |\zeta_k ~\text{is adapted to }~ \sigma(\calI^{S,\omega}_t,~ t = 0,1, \cdots, k)\}$. In the above, the set $\Xi^{dec}$ denotes the space of admissible \textit{decentralized} sensor policies and $\calI^{S,\omega}_k = \{X_{0:k},y_{0:k-1}, U_{0:k-1},\gamma_{0:k-1}\}$ defines the information structure of the representative team's sensor. We note here that the above set $\calI^{S,\omega}$ constitutes only the \textit{local signal information of the representative team} as opposed to the information set defined in \eqref{eq:info_cent}. Moreover, the deterministic term ${g} := ({g}_0, {g}_1, \cdots, {g}_{T-1})$ represents the mean-field trajectory and can be viewed as an infinite-team approximation to the channel utilization ratio term in \eqref{Scheduling_Cost}. While this assertion is currently heuristic, we will rigorously establish this argument later (in Lemma \ref{lem:Approx_MFE}), demonstrating that under specific assumptions in the large-population regime, the mean-field trajectory provides an approximation to the empirical channel utilization term (as in \eqref{Scheduling_Cost}).
Finally, the controller $C(\omega)$ applies a control policy which incurs a running cost of
\begin{align}\label{control_cost}
    c_2^{\omega}(X_k, U_k) := \|X_k\|^2_{Q(\omega)} + \|U
    _k\|^2_{R(\omega)},
\end{align}
where the control inputs $U_k$ at instants $k$ are chosen from the policy $\pi := \{\pi_k\}_{\forall k}$, which lies in the space of admissible control policies defined similarly to $\!\mathcal{M}^i\!$ as before.
Thus, the representative team's optimization problem is defined as follows:

\begin{problem}[Representative team problem]\label{Prob:Overall_MF}
    \begin{align}\label{Team_CostMF}
    & \hspace{1cm} \inf_{(\pi, \zeta)\in\mathcal{M} \times \Xi^{dec}} J^\omega(\pi,\zeta,g) := \frac{1}{T} \mathbb E \left[\sum_{k=0}^{T-1}  c_2^\omega(X_k,U_k) + \lambda c_1^\omega(\gamma_k, g_k) \right] \nonumber \\
    & \mbox{s.t.} ~\eqref{decoder_signal} \text{ and } \eqref{systemMF} \text{ hold},
\end{align}
where $J^\omega(\pi,\zeta,g)$ denotes the joint time-average overall cost incurred by the representative team with $ c_1^{\omega}(\cdot,\cdot)$ and $c_2^\omega(\cdot, \cdot)$ defined, respectively, by (\ref{Scheduling_CostMF}) and (\ref{control_cost}). 
\end{problem}
The minimization problem above can be viewed as a team-optimal control problem of the representative team of type $\omega$ for a given mean-field trajectory $g$. In order to construct its solution, we invoke the following assumption on the mean-field trajectory, which is natural for (and does not bring in any loss of generality to) the problem at hand.
\begin{assumption}\label{As_ctrb}
    The MF trajectory ${g} \!\in\! \ell_{\infty}^T$, where we define $\ell_{\infty}^T\!:= \!\{{g}_k \!\in \!\mathbb{R} \mid \sup_{ k \in [T]} $ $\!| {g}_k|\! \leq 1\}$ as the set of bounded $T \!$--length sequences.    
    
\end{assumption}

Next, in order to define the mean-field team equilibrium (MFTE), we introduce the following two operators \cite{saldi2018markov}:
\begin{enumerate}
    \item $\Phi:\! \ell_{\infty}^T \!\rightarrow\! 2^{\mathcal{M} \times \Xi^{dec}}$, which, for any given mean-field trajectory $g$, outputs one or more team optimal policies (given by $ \argmin_{(\pi, \zeta)\in\mathcal{M} \times \Xi} J^\omega(\pi,\zeta,{g})$), and
    \item $\Psi: \mathcal{M} \times \Xi^{dec} \rightarrow \ell_\infty^T$, which maps $(\pi,\zeta) \mapsto {g}$ and generates a MF trajectory $g$ from a given policy ($\pi,\zeta$).
\end{enumerate}
    Consequently, we can define the MFTE as follows.

    \begin{definition}[Mean-Field Team Equilibrium]
        A mean-field team equilibrium (MFTE) is a policy-trajectory pair $((\pi^*,\zeta^*), {g}^*)$ such that $(\pi^*, \zeta^*) \in \Phi(g^*)$ and $g^* = \Psi((\pi^*,\zeta^*))$.
    \end{definition}

In the sequel, we first compute the team optimal policy pair ($\pi^*,\zeta^*$) solving Problem \ref{Prob:Overall_MF} for a given MF trajectory $g$ and then provide sufficient conditions for the existence of a unique approximate MFTE.
\vspace{-0.2cm}
\section{Mean-Field Team Equilibrium Analysis}\label{sec:eqb_analysis}
\vspace{-0.15cm}
We start this section by solving Problem \ref{Prob:Overall_MF}. Using completion of squares, we can re-write the cost function in \eqref{Team_CostMF} as
\begin{align}\label{eq:comp_square}
    J^\omega(\pi,\zeta,&g)  =  \mathbb E [\|X_0(\omega)\|^2_{P_0(\omega)}] + \frac{1}{T} \sum_{k=0}^{T-1} tr(P_k(\omega)K_W(\omega)) \nonumber \\
    & \hspace{-1cm} + \frac{1}{T} \mathbb E \Big[ \sum_{k=0}^{T-1}\|U_k + \hat R_k^{-1}(\omega) B^\top(\omega) P_{k+1}(\omega)A(\omega)X_k\|^2_{\hat R_k(\omega)} \nonumber \\
    & + \lambda \gamma_k {g}_k  \Big],
\end{align}
where $\hat R_k(\omega):= R(\omega)+B^\top(\omega) P_{k+1}(\omega)B(\omega)$ and $P_k \succeq 0$ obeys the Riccati difference equation (RDE)
\begin{align}\label{eq:ARE_P}
    P_k(\omega) & = Q_k(\omega) + A^\top(\omega)P_{k+1}(\omega)A(\omega) \nonumber \\
    &  - A^\top(\omega) P_{k+1} B(\omega) \hat R_k^{-1}(\omega) B^\top(\omega) P_{k+1}(\omega) A(\omega) \nonumber \\
    P_T(\omega)  & = 0.
\end{align}
 Then, we have the following result on the optimal controller policy ($\pi^*$).

\begin{figure*}[t]
	\centerline{\includegraphics[width=0.7\textwidth]{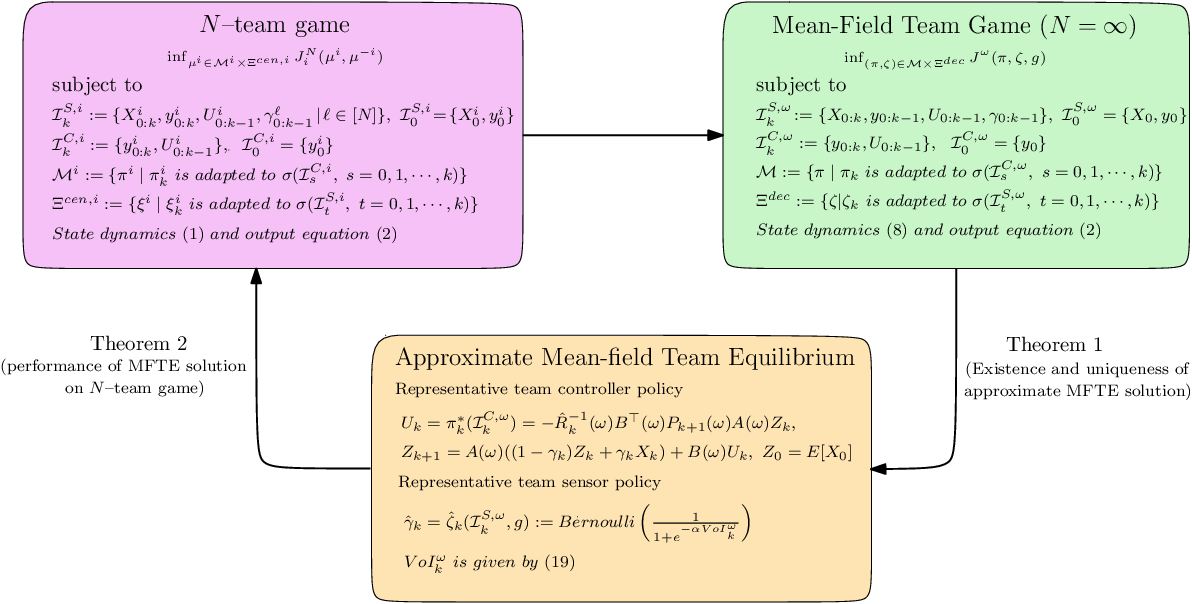}}
	\caption{\small{The flowchart above depicts our proposed mean-field team game approach. Precisely, we start with an $N$--team game, and then propose a MFT game. Subsequently, we present Theorem \ref{thm:uniqueness} to prove the existence of a unique approximate MFTE solution of the MFT game. Finally, we show that this solution constitutes an approximate Nash equilibrium solution for the $N$--team game using Theorem \ref{thm:eps_Nash}.}}
	\label{Fig:flowchart}
    \vspace{-0.4cm} 
\end{figure*}

\begin{proposition}\label{prop:MMSE}
Suppose $P_k (\omega) \succeq 0$ satisfies the RDE \eqref{eq:ARE_P} for all $k \in [T]$ for a representative team of type $\omega$, for all $\omega \in \Omega$. Then, the optimal controller is unique and is obtained as
\begin{align}\label{eq:opt_controller}
    U_k = \pi^*_k(\mathcal{I}^{C,\omega}_k) = -\hat{R}^{-1}_k(\omega)B^\top(\omega) P_{k+1} (\omega) A(\omega) Z_k,
\end{align}
$\forall \omega \in \Omega, \forall k \in [T]$, with $Z_k \!:= \!\mathbb E[X_k\! \mid\! \mathcal{I}^{C,\omega}_k]$ being the conditional mean of the state given the information available to the controller. Furthermore, $Z_k$ follows the dynamics
    \begin{align}
        Z_{k+1} & = A(\omega)((1-\gamma_k)Z_k+\gamma_k X_k)  + B(\omega)U_k, 
    \end{align}
$\forall \omega \in \Omega, \forall k \in [T]$ with $Z_0 = \EE[X_0]$. \hfill \qed
\end{proposition}

The proof of the above proposition readily follows from \eqref{eq:comp_square} and by taking expectation in \eqref{systemMF} conditioned on the controller information structure and is thus not included here.

Having determined the optimal controller policy, we now move on to computing the optimal sensor policy for each representative team of type $\omega$. To facilitate this, let us define the estimation error $e_k:= X_k - Z_k$. Then, by substituting the optimal controller \eqref{eq:opt_controller} in \eqref{eq:comp_square}, we arrive at
\begin{align}\label{eq:sensing_cost}
    J^\omega(\zeta,g) & =  J^* + \frac{1}{T} \mathbb E \Big[ \sum_{k=0}^{T-1} \|e_k\|^2_{\Gamma_k} + \lambda {g}_k \gamma_k  \Big],
\end{align}
where $J^* \! \!:=\! (\nicefrac{1}{T})\!\sum_{k=0}^{T-1} \!tr(P_k(\omega)K_W(\omega)) + \mathbb E [\|X_0(\omega)\|^2_{P_0(\omega)}] $,

$\Gamma_k(\omega):= A^\top(\omega) P_{k+1}^\top (\omega) B(\omega) \hat R^{-1}_k(\omega) B^\top (\omega) P_{k+1}(\omega) A(\omega) $, and $e_k$ satisfies the following stochastic difference equation:
\begin{align}\label{eq:est_error}
  e_{k+1} = (1-\gamma_k)A(\omega)e_k + W_k. 
\end{align}
Thus, the sensor objective is to solve for the optimal policy minimizing \eqref{eq:sensing_cost} subject to \eqref{eq:est_error}. To obtain it, let us first define the optimal cost-to-go ($V_k^\omega(\mathcal{I}^{S,\omega}_k,g)$) of the sensor as follows:
\begin{align*}
V_k^\omega(\mathcal{I}^{S,\omega}_k,g):= \min_{\zeta \in \Xi^{dec} \mid \pi = \pi^*} \frac{1}{T}\mathbb E \Big[ \sum_{t=k}^{T-1} \|e_t\|_{\Gamma_t}^2+ \lambda g_t \gamma_t \mid \calI^{S,\omega}_k \Big],
\end{align*}
for all $k\in [T]$ and with the convention that $g_{-1} = 0$. Then, we invoke the following definition \cite{soleymani2021valu} of the value of information (VoI), which quantifies the trade-off between the benefit gained from transmitting a data packet over the channel and the associated cost of transmission.

\begin{definition}
    The VoI at instant $k$ for the sensor of team of type $\omega$ is defined as a function of $\mathcal{I}^{S,\omega}_k$ as
    \begin{align}
        {\operatorname{VoI}}_k^{\omega} := V_k^{\omega}(\mathcal{I}^{S,\omega}_k,g)|_{\gamma_k = 0} - V_k^{\omega}(\mathcal{I}^{S,\omega}_k,g)|_{\gamma_k = 1},
    \end{align}
    where $V^{\omega}(\cdot,\cdot)|_{\gamma}$ denotes the sensor's optimal cost-to-go under the sensor action  $\gamma$.
\end{definition}

Subsequently, we present the following proposition which characterizes the optimal sensor policy ($\zeta^*$).

\begin{proposition}\label{Prop:Optimal_policy}
   Consider a representative team of type $\omega \in \Omega$. Then, its optimal sensing actions are chosen according to the optimal sensor policy $\zeta^*$, which has a threshold structure, and given as
    \begin{align}
        \gamma_k = \zeta^*_k(\mathcal{I}^{S,\omega}_k,g)  = \mathbb I [\operatorname{VoI}_k^\omega \ge 0],~~ \forall k \in [T],
    \end{align}
    where we have
    \begin{small}
    \begin{align}\label{eq:cond_VF}
        \operatorname{VoI}_k^\omega & = e_k^\top A(\omega)^\top \Gamma_k(\omega)A(\omega)e_k - \lambda {g}_k + \varepsilon_k(\omega), \\
        \varepsilon_k^{\omega} & := \mathbb{E}[V^\omega_{k+1}(\mathcal{I}^{S,\omega}_{k+1},g) \!\mid \! \mathcal{I}^{S,\omega}_k,\gamma_k = 0] - \mathbb{E}[V^\omega_{k+1}(\mathcal{I}^{S,\omega}_{k+1},g) \! \mid \! \mathcal{I}^{S,\omega}_k,\gamma_k = 1] \nonumber,
    \end{align}
    \end{small}
    and $\Gamma_k(\omega)$ is as defined in \eqref{eq:sensing_cost}.
\end{proposition}
\begin{proof}
    We start by noting that for a fixed mean-field trajectory $g$, minimization of the objective function \eqref{eq:sensing_cost} under \eqref{eq:est_error} represents a stochastic optimal control problem of the representative team of type $\omega$, where its decision policy ($\zeta$) depends only on the local information and the given trajectory $g$. Hence, for each representative team of type $\omega \in \Omega$, the proof follows from \cite[Theorem 1]{soleymani2021valu}, and is thus complete.
\end{proof}
As a remark, we emphasize that \textit{for a given $g$}, the optimization Problem \ref{Prob:Overall_MF} is a team-optimal control problem for each representative team of type $\omega \in \Omega$. This decoupling is enabled by the mean-field approach, allowing us to solve for each representative team's stochastic optimal control problem independently for a fixed $g$ (in Proposition \ref{prop:MMSE} and \ref{Prop:Optimal_policy}). Our next step is to verify the consistency of this solution. Specifically, if each team adopts its optimal policy $(\pi^*, \zeta^*)$, the collective behavior should regenerate the mean-field trajectory $g$ that was initially assumed. If this condition holds, then, we would have obtained \textit{the MF team equilibrium policy and the corresponding equilibrium trajectory}.

In this regard, let us first make a crucial observation. The optimal sensor policy $\zeta^*$ computed in Proposition \ref{Prop:Optimal_policy} is discontinuous at the point where $\operatorname{VoI}_k^\omega =0$. This makes it prohibitive to employ contraction-type arguments to establish the existence of a unique
MFTE \cite{huang2007large,cui2021approximately}. Thus, our objective in what follows is to construct an approximation of the policy $\!\zeta^*\!$, by restricting it to a special class of policies, then leading to a unique (approximate) MFTE within that class.

Let $\alpha>0$, and define a new policy $\hat{\zeta}:=\zeta^\alpha$ from which an action $\hat{\gamma}_k$ is chosen as follows:
\begin{align}\label{approx_scheduling_policy}
    \hat{\gamma}_k & = \hat{\zeta}_k(\mathcal{I}^{S,\omega}_k,g):= \begin{cases}
        1, &  \text{w.p. } p_k:=  \frac{1}{1+ e^{-\alpha \operatorname{VoI}^\omega_k}}, \\
        0, & \text{w.p. } 1-p_k.
    \end{cases}
\end{align}
We note that the above policy is indeed an approximation of $\zeta^*$ since $\hat{\zeta} \rightarrow \zeta^*$ as $\alpha \rightarrow \infty$ under the convention that $\hat{\gamma}_k = 1$ whenever $\operatorname{VoI}_k^\omega = 0$. Further,  it is also measurable with respect to $\mathcal{I}^{S,\omega}_k$.
Then, we can define the mean-field operator $\mathcal{T}$ as
\begin{align}\label{MF_operator}
    \hat{g}_{k+1} & = [\mathcal{T}(\hat{g})]_k  := \mathbb{E}_{\hat{\zeta},\mathbb P(\cdot)}[\hat{\gamma}_k] \!=\! \mathbb{E}_{\mathbb P(\cdot)}\left[\frac{1}{1+ e^{-\alpha \operatorname{VoI}^\omega_k}}\right].
\end{align}
We next prove the existence (and subsequently the uniqueness) of an approximate MFTE in the sense above, which is a pair $((\pi^*,\hat{\zeta}^*), \hat{g}^*)$, such that a given $\hat g^*$ generates the policy pair $(\pi^*, \hat \zeta^*)$ using \eqref{eq:opt_controller} and \eqref{approx_scheduling_policy}, and the policy $(\pi^*,\hat{\zeta}^*)$ regenerates $\hat g^*$ using the expectation function in the second equality in \eqref{MF_operator}. Equivalently, we prove that the pair $((\pi^*,\hat{\zeta}^*), \hat{g}^*)$ is such that $\hat{g}^*$ is a fixed point of the operator $\mathcal{T}$.

\begin{proposition}\label{prop:existence}
    Under Assumption \ref{As_ctrb}, there exists an $\alpha$--approximate MFTE for any $\alpha>0$.
\end{proposition}
\begin{proof}
    A detailed proof is provided in Appendix I. In summary, the proof relies on first proving that the mean-field operator $\mathcal{T}$, as defined in \eqref{MF_operator}, is a continuous function in $\hat g$. Subsequently, Brouwer's fixed point theorem is applied to show the existence of at least one MFTE for any $\alpha >0$.
\end{proof}

In addition, we can guarantee the contraction of the MF operator $\mathcal{T}$ for sufficiently low values of $\alpha$, which allows us to prove a stronger result on the uniqueness of the approximate MFTE as follows.

\begin{theorem}\label{thm:uniqueness}
    Suppose Assumption \ref{As_ctrb} holds. Then, the mean-field operator $\mathcal{T}$ is a contraction, if 
    \begin{align}\label{contraction_condition}
        2 \alpha \lambda T < 1.
    \end{align}
    Furthermore, under condition \eqref{contraction_condition}, there exists a unique $\alpha$--approximate MFTE, that is the fixed point of the operator $\!\mathcal{T}\!$.
\end{theorem}
\begin{proof}
    A detailed proof is provided in Appendix I. Briefly, we use the contractivity of $\mathcal{T}$ under the condition \eqref{contraction_condition} to employ Banach's fixed point theorem to show the existence of a unique $\alpha$--approximate MFTE.
\end{proof}

The above results show the existence of an approximate MFTE which is unique if the parameter $\alpha$ is small enough. Further, we would also like to note from the proof of Theorem \ref{thm:uniqueness} that the contraction condition in \eqref{contraction_condition} is only a sufficient one and one may be able to obtain a more precise characterization of the `smallness' of $\alpha$. Furthermore, in the numerical simulation section, we will compute an approximate MFTE even when  \eqref{contraction_condition} is not satisfied.

Now that we have characterized the approximate MFTE for the mean-field game, we next proceed toward proving that this MFTE satisfies the approximate Nash property (as formalized later in Section \ref{sec:eps_nash}) for the original finite-team game problem.

\section{Performance of Mean-Field Approximation}\label{sec:perf_approx}
In this and the subsequent section, we return to the finite-team setup of Section \ref{sec:formulation} and show that the $\alpha$--approximate MFTE as computed in the previous section constitutes an ($\epsilon_a,\epsilon_b$)--Nash equilibrium for the finite-team game problem. The term $\epsilon_a$, as we will see, arises as a result of the approximation parameter $\alpha$ in the sensor policy and the term $\epsilon_b$ term arises due to the infinite team approximation. 

Suppose now that team $i$ deviates toward using a sensor policy $\xi^i \in \Xi^{cen,i}$ instead of the $\alpha$--approximate MFTE policy $\hat \zeta^*$. Moreover, let us denote an action chosen from $\xi^i$ at instant $k$ as $\tilde \gamma_k$ and define the empirical average of the sensor actions at instant $k$ as $\gamma^{N,av}_k:= (\nicefrac{1}{N})\sum_{j=1}^N\gamma^j_k$. Consequently, we prove the following lemma which computes the mean deviation between the empirical average $\gamma^{N,av}$ and the approximate equilibrium mean-field trajectory $\hat g^*$. This formalizes the heuristic assertion made at the beginning of the previous subsection \ref{subsec:MFTG} regarding the introduction of the mean-field term $g$ in \eqref{Scheduling_CostMF}.

\begin{lemma}\label{lem:Approx_MFE}
    Let team $i$ employ the sensor policy $\xi^i \in \Xi^{cen,i}$ while all the other teams employ the equilibrium sensor policy $\hat \zeta^*$. Then, we have that
    \begin{align}\label{eq:MFE_approx}
        \frac{1}{T} \sum_{k=0}^{T-1} \mathbb E\Big[\left|\gamma^{N,av}_k - \hat g^*_k \right|\Big] = \mathcal{O}\Big( \frac{1}{\sqrt{N}} + \epsilon_{P,N} \Big),
    \end{align}        
    where $\epsilon_{P,N} = \sum_{\omega \in \Omega }| \mathbb{P}_N(\omega) - \mathbb P(\omega) |$.
\end{lemma}
\begin{proof}

Consider the following inequality:
\begin{small}
\begin{align}\label{Approx_term}
    & \sum_{k=0}^{T-1} \mathbb{E}\Big[ \Big|\frac{1}{N}\sum_{j=1}^N\gamma^j_k - \hat{g}^*_k \Big|\Big]
    \leq \underbrace{\sum_{k=0}^{T-1} \mathbb{E}\Big[ \Big|\frac{1}{N}\sum_{j=1}^N(\gamma^j_k - \hat{\gamma}^j_k) \Big|\Big]}_{\text{=:Term 1}} +\!\! \underbrace{\sum_{k=0}^{T-1}\!\mathbb{E}\Big[  \Big| \!\frac{1}{N}\sum_{j=1}^N\hat \gamma^j_k \!-\! \mathbb{E}[\hat g^{*,j}_k] \Big|\Big]}_{\text{=:Term 2}} \! +\! \underbrace{\sum_{k=0}^{T-1}\mathbb{E} \Big[\Big| \!\frac{1}{N}\!\sum_{j=1}^N\mathbb{E}[ \hat g^{*,j}_k] \!-\! \hat g^*_k \Big|\Big]}_{\text{=:Term 3}}.
\end{align}
\end{small}
We now bound each of the three terms in \eqref{Approx_term} below. To that end, let the action chosen by the sensor of team $i$ be given as $\tilde \gamma^{i}_k \sim \xi^{i}_k(\mathcal{I}^{S,i}_k)$ with $\mathcal{I}^{S,i}_k$ defined in \eqref{eq:info_cent}. Then, Term 1 in \eqref{Approx_term} can be bounded as follows:
\begin{align}\label{perf:B1}
    \text{Term 1} & = \sum_{k=0}^{T-1} \mathbb{E}\Big[ \Big|\frac{1}{N}\sum_{j \in [N]\setminus \{i\}}(\hat \gamma^j_k - \hat{\gamma}^j_k) + \frac{1}{N} (\tilde \gamma^i_k - \hat{\gamma}^i_k)\Big|\Big] \nonumber \\
    & = \sum_{k=0}^{T-1} \mathbb{E}\Big[ \Big|\frac{1}{N}\sum_{j \in [N]\setminus \{i\}}(\hat \zeta^{*,j}_k(\mathcal{I}^{S,j}_{k}) - \hat{\zeta}^{*,j}_k(\mathcal{I}^{S,j}_{k})) + \frac{1}{N} (\tilde \gamma^{i}_k - \hat{\gamma}^i_k)\Big|\Big] \nonumber \\
    & = \sum_{k=0}^{T-1} \frac{1}{N}\mathbb E [|\tilde \gamma^i_k - \hat \gamma^i_k|] \leq \frac{2T}{N},
\end{align}
where the second equality follows since teams other than the $i^{th}$ one are still following the policy $\hat \zeta^*$.

We next bound Term 2 in \eqref{Approx_term}. To that end, let us define the quantity $\bar g^{*,j}_k := \gamma^j_k - \mathbb{E}[\hat g^{*,j}_k]$. Then, we have that
\begin{align}\label{eq:temp_term2}
    & \mathbb{E}\Big[ \frac{1}{N}\sum_{j=1}^N\hat \gamma^j_k - \mathbb{E}[\hat g^{*,j}_k] \Big]^2  = \frac{1}{N^2}\mathbb{E}\Big[ \sum_{j \in [N]} \bar  g^{*,j}_k\Big]^2 \nonumber \\
    & \!= \! \frac{1}{N^2}\mathbb{E}\Big[ \sum_{j,\ell \in [N]} \bar g^{*,j}_k \bar g^{*,\ell}_k\Big] \!=\! \frac{1}{N^2}\mathbb{E}\Big[ \sum_{j \in [N]} (\bar g^{*,j}_k)^2 \Big] \! \le\! \frac{1}{N},
\end{align}
where the third equality follows since $\bar g^{*,j}_k$ and  $\bar g^{*,\ell}_k$ are uncorrelated because the system noise $W_k^i$ is independent across different teams, which leads to decoupling between the corresponding sensor policies.

Thus, it follows from \eqref{eq:temp_term2} that:
\begin{align}\label{perf:B2}
    \text{Term 2} \leq \sum_{k=0}^{T-1} \frac{1}{\sqrt{N}} \leq \frac{T}{\sqrt{N}}.
\end{align}

Finally, since ${\hat g}_k^{*,\omega}$ takes values in $[0,1]$, Term 3 in \eqref{Approx_term} can be bounded as follows:
\begin{align}\label{perf:B3}
    \text{Term 3} & = \sum_{k=0}^{T-1}\Big[ \sum_{\omega \in \Omega} \hat g^{*,\omega}_k | \mathbb{P}_N(\omega) - \mathbb P(\omega) | \Big] \nonumber \\
    & \leq T \sum_{\omega \in \Omega }| \mathbb{P}_N(\omega) - \mathbb P(\omega) |.
\end{align}

Using \eqref{perf:B1}, \eqref{perf:B2}, and \eqref{perf:B3} in \eqref{Approx_term}, we arrive at the result in \eqref{eq:MFE_approx}. The proof is thus complete.
\end{proof}

The above lemma demonstrates that the deviation of the empirical average from the approximate mean-field trajectory vanishes toward zero as the number of teams increases, and hence the latter `reasonably' approximates the former if the number of teams is large.

\section{Approximate Nash Equilibrium Analysis}\label{sec:eps_nash}

In this subsection, we prove the approximate-Nash property of the MF solution for the finite-team game problem, which is made precise below.

\begin{definition}\label{defn:epsNash}
    The set of team policies $\{\mu^{*,i} \in \mathcal{M}^i \times \Xi^{cen,i}, ~i \in [N]\}$ constitutes an $\epsilon$--Nash equilibrium for the cost functions $J^N_i$ for $i \in [N]$, if there exists $\epsilon >0$ such that
    \begin{align*}
        J^N_i(\mu^{*,i},\mu^{*,-i}) \leq \inf_{\mu^i \in \mathcal{M}^i \times \Xi^{cen,i}} J^N_i(\mu^{i},\mu^{*,-i}) + \epsilon, ~ \forall i \in [N].
    \end{align*}
\end{definition}
We also note that the infimum on the RHS of the above inequality is taken over the set of centralized equilibrium policies as defined before, hence bringing in an element of conservatism.

Let us start by defining the following cost functions:
\begin{subequations}
\begin{small}
    \begin{align}
    & \!\!J_i^N(\pi^{*,i},\hat{\zeta}^{*,i},\hat{\zeta}^{*,-i})  = \frac{1}{T} \mathbb{E}\left[ \sum_{k=0}^{T-1} \|\hat{X}^i_k\|^2_{Q(\omega_i)}\!\! +\! \|U^i_k\|^2_{R(\omega_i)}\!\! +\! \lambda \hat{\gamma}^{N,av}_k \hat{\gamma}^i_k\right],\! \label{costa}\\
     &\!\!J^\omega(\pi^*\!,\hat{\zeta}^*\!,\hat{g}^*) \!\!= \!\!\frac{1}{T} \mathbb{E}\!\left[ \sum_{k=0}^{T-1} \!\|\hat{X}_k\|^2_{Q(\omega)} \!\!+\! \|U_k\|^2_{R(\omega)} \!\!+\! \lambda \hat{g}^*_k \hat{\gamma}_k\right]\!,\! \label{costb}\\
    & \!\!J_i^N(\pi^{*,i},\xi^i, \hat{\zeta}^{*,-i}) =  \frac{1}{T} \mathbb{E}\left[ \sum_{k=0}^{T-1} \!\|{X}^i_k\|^2_{Q(\omega_i)}\! \!+\! \|U^i_k\|^2_{R(\omega_i)}\! \!+\! \lambda \tilde \gamma^{N,av}_{k} \tilde \gamma^i_k\right]\!,\!\! \label{costc}\\
    & \!\!J^\omega(\pi^*\!,{\zeta}^d\!,\hat{g}^*) \!\!=\!\! \frac{1}{T} \mathbb{E}\left[ \sum_{k=0}^{T-1} \!\!\|X^d_k\|^2_{Q(\omega)}\!\! + \!\|U_k\|^2_{R(\omega)}\!\! +\! \lambda \hat{g}^*_k \gamma^{d}_k\right]\!.\!\!\! \label{costd}
\end{align}
\end{small}
\end{subequations}
The cost \eqref{costa} is the one incurred by team $i$ in the finite-team setting when all the players follow the $\alpha$--approximate MFTE policy. The cost in \eqref{costb} is the representative team's cost under the approximate equilibrium policy and equilibrium mean-field. The cost in \eqref{costc} is incurred by team $i$ when it deviates toward using a policy $\xi^i \in \Xi^{cen,i}$ while all other teams still employ the MF policy. Finally, the cost in \eqref{costd} is incurred when a team follows a \textit{decentralized} policy $\zeta^d \in \Xi^{dec}$ under the approximate equilibrium MF trajectory. The notations of the states and actions in the above costs are under the respective policies and the trajectory distributions, and are self-explanatory.

Further, let us define $[\bar T] \subseteq [T]$ to be the set of all instants $k \in [T]$ for which $\operatorname{VoI}_k \ne 0$. Then, we state and prove another main result of the paper in the following theorem, which essentially characterizes how `well' the MFTE policy performs on the $N$--team system. 

\begin{theorem}\label{thm:eps_Nash}
    Suppose that $X_0$ is given. Then, the set $\{(\pi^{*,i}, \hat{\zeta}^{*,i}), i \in [N]\}$ of decentralized team policies constitute $(\epsilon_a,\epsilon_b)$--Nash equilibrium for the finite-team uplink scheduling game, i.e.,
\begin{align*}
& J_i^N(\pi^{*,i},\hat{\zeta}^{*,i}, \hat{\zeta}^{*,-i})  \leq  \inf_{\xi^i \in \Xi^{cen,i}} J_i^N(\pi^{*,i},\xi^i, \hat{\zeta}^{*,-i}) + \epsilon_a(\alpha) + \epsilon_b(N),
\end{align*}
where $\!\epsilon_a(\alpha) \!\!=\!\! \mathcal{O}\Big(\! \frac{1}{1\!+\!\min\limits_{k \in [\bar T]}\{e^{\alpha | \operatorname{VoI}_k\!|}\}}\!\Big)$ and $\!\epsilon_b(N) \!\!= \!\mathcal{O}\Big( \!\frac{1}{\sqrt{N}} + \epsilon_{P,N}\!\Big)$.
\end{theorem}

\begin{proof}
To establish the inequality in Theorem \ref{thm:eps_Nash}, consider the following:
\begin{align}
    & J_i^N(\pi^{*,i},\hat{\zeta}^{*,i},\hat{\zeta}^{*,-i}) - \inf_{\xi^i \in \Xi^{cen,i}} J_i^N(\pi^{*,i},\xi^i, \hat{\zeta}^{*,-i}) \nonumber
\end{align}
\begin{subequations}
    \begin{align}
     & = J^N_i(\pi^{*,i}, \hat{\zeta}^{*,i},\hat{\zeta}^{*,-i}) - J^\omega(\pi^{*},\hat{\zeta}^*,\hat{g}^*) \label{first_term}\\
    &  + J^\omega(\pi^{*},\hat{\zeta}^*,\hat{g}^*) - \inf_{\zeta^d \in \Xi^{dec}} J^\omega(\pi^{*},{\zeta^d},\hat{g}^*) \label{second_term}\\
    &  + \!\!\inf_{\zeta^d \in \Xi^{dec}} \!J^\omega(\pi^{*},{\zeta^d},\hat{g}^*) - \!\!\inf_{\xi^i \in \Xi^{cen,i}} \!J_i^N(\pi^{*,i},\xi^i, \hat{\zeta}^{*,-i}). \label{third_term}
\end{align}
\end{subequations}
The proofs for bounding each of the above terms are provided in Appendices II and III. The term in \eqref{first_term} is bounded by Lemma \ref{eq:First_term_bound} as in Appendix II. Next, to bound the term in \eqref{second_term}, we first prove the boundedness of the state action value function as in Lemma \ref{lem:Q_bounded} in Appendix III and consequently use Proposition \ref{prop:Second_term_bound} as in Appendix III. Finally, the term in \eqref{third_term} is bounded by Lemma \ref{eq:Third_term_bound} as in Appendix II.

Combining the results from Lemmas \ref{eq:First_term_bound}, \ref{eq:Third_term_bound} and Proposition \ref{prop:Second_term_bound}, we arrive at
\begin{align}
    & J_i^N(\pi^{*,i},\hat{\zeta}^{*,i}, \hat{\zeta}^{*,-i}) - \inf_{\xi^{i} \in \Xi^{cen,i}} J_i^N(\pi^{*,i},\xi^{i}, \hat{\zeta}^{*,-i}) \nonumber \\
    & \hspace{0.5cm} \leq \epsilon_a(\alpha) + \mathcal{O}\Big(\frac{1}{\sqrt{N}} + \epsilon_{P,N} \Big) = \epsilon_a(\alpha) + \epsilon_b(N).
\end{align}
This completes the proof of the theorem.

\end{proof}

Theorem \ref{thm:eps_Nash} states that the $\alpha$--approximate MFTE policy derived using the infinite agent system behaves as an approximate Nash equilirium solution for the original finite-team game, where the approximation results from 1) using the policy $\hat \zeta^*$ instead of the optimal sensing policy $\zeta^*$, and 2) using infinite-team system as a proxy for the $N$--team system. This final result completes our solution to Problem \ref{Prob:Stoch_min_Overall} using the MFTG approach as illustrated in Fig. \ref{Fig:flowchart}.

\section{Numerical Experiments}\label{sec:num_exp}

In this section, we present two numerical simulations to evaluate the performance of the MFTE solution derived in Section \ref{sec:eqb_analysis}. Throughout, we consider the type set to be singleton, i.e., all the teams are homogeneous.

\begin{example}
   In the first example, we consider teams with open-loop unstable scalar plants and parameters: $A = 1.5,$ $ B = 2,$ $ Q=2,$ $ R=1,$ $ \Sigma_0 = 0.1,$ $ K_W = 0.3$ and a time horizon of $T=51$. We then compute the controller policy by solving  the Riccati equation \eqref{eq:ARE_P} and subsequently, the sensor policy by using Proposition \ref{Prop:Optimal_policy} and \eqref{approx_scheduling_policy}. Finally, we compute the approximate MFTE by iterated application of the MF operator $\mathcal{T}$ as in \eqref{MF_operator}. Next, we note that since the approximate sensing policy in \eqref{approx_scheduling_policy} depends on the value of information, which in turn relies on the conditional state value function, we employ supervised learning with linear function approximation \cite{bertsekas1996neuro} to approximate it. Specifically, we start by parametrizing the state value function as $V_k(e,g) = \theta_{k}^\top \phi_1(e) + \phi_2(g)$, where $\theta_{k}$ represent learnable parameters and $\phi_1(\cdot), \phi_2(\cdot)$ denote the basis functions which are a function of the error $e$ and the mean-field $g$, respectively. We then apply value iteration using this parametrized function and refine the coefficients through supervised learning, minimizing the mean-squared error loss. The process terminates once convergence is achieved within a predefined threshold $\delta > 0$.
   
   For purposes of the first three studies, we will consider polynomial basis functions of degree 2, i.e., $\phi_i(x) =[1,~x,~x^2]$.
   Later in study 4), we also analyze the impact of the degree of the basis functions on the average cost of each team.

 \begin{figure}[h]
 \vspace{-3mm}
\centerline{\includegraphics[width=0.7\columnwidth]{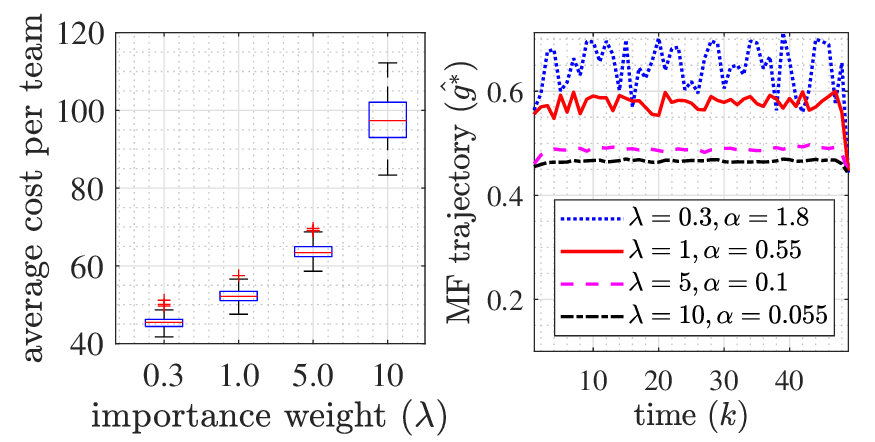}}
	\caption{\small{The left subfigure shows the variation in average cost per team with the importance weight $\lambda$ and the right subfigure shows the variation of the corresponding equilibrium MF trajectory with time.}}
	\label{Fig:costVslambda}
    \vspace{-0.4cm} 
\end{figure}

\textbf{1) Variation of average cost vs $\mathbf{\lambda}$ and $\mathbf{\hat g^*}$ vs time: }In the first study, we plot the variation of the average cost per team versus the importance weight $\lambda$ in Figure \ref{Fig:costVslambda} (left subfigure) for four different values of $\lambda = 0.3, 1.0, 5.0$ and $10$ with corresponding values of $\alpha$ as $1.8, 0.55, 0.1$ and $0.055$. The
figure shows a box plot depicting the median (red line) and
spread (box) of the average cost per team over 100 runs for
each value of $\lambda$ and $N=50$ teams, under the MFTE policy. From the same, we observe that the average cost increases with the increase in $\lambda$ which results from an increase in overall estimation error due to increase in the price of communication $\lambda$. Alongside (in the right subfigure), we show the variation of equilibrium MF trajectory $\hat g^*$ with time. From the same, we observe a decrease in channel utilization with the price of communication increases, aligned with intuition.

\textbf{2) Variation of average cost vs Boltzman parameter $\mathbf{\alpha}$: }Next, in Figure \ref{Fig:CostFixedLambda} (left subfigure), we analyze the variation of the average cost per team under the MFTE policy for different values of $\alpha$, using a box plot with $N=50$ and $\lambda = 5$. 
The results indicate a decreasing trend in average cost as $\alpha$ increases. This behavior can be attributed to the fact that a higher value of $\alpha$ reduces the level of approximation in constructing the policy $\hat \zeta$ in \eqref{approx_scheduling_policy} from the optimal policy $\zeta$ thereby leading to improved policy accuracy and a lower overall cost.


\begin{figure}[h]
	\centerline{\includegraphics[width=0.7\columnwidth]{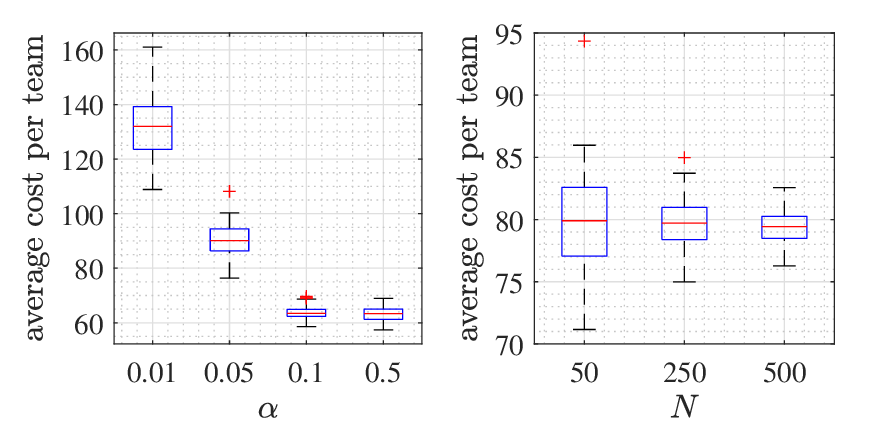}}
	\caption{\small{The left plot shows the variation of average cost per team vs the Boltzmann parameter ($\alpha$) for a fixed $\lambda = 5$; the right plot shows the variation of the average cost per team vs $N$.}}
	\label{Fig:CostFixedLambda}
    \vspace{-0.4cm} 
\end{figure}

\textbf{3) Variation of average cost vs $N$: }Next, in Figure \ref{Fig:CostFixedLambda} (right subfigure) we present box plots for the average cost per team over 100 runs for different values of $N$ and $\alpha = 0.1$ and $\lambda = 5$. The results show that as $N$ increases, the spread of the box plot decreases, indicating that the mean-field approximation becomes more accurate with a larger number of teams. This aligns with the theoretical expectation that the MFTG formulation provides a better approximation as the system size grows.

\textbf{4) Variation of average cost vs degree of basis functions: }Finally, in Figure \ref{Fig:CostVsDegree}, we study the effect of increasing the number of basis functions in $\phi_1(\cdot)$ (used in function approximation) on the average cost of the teams. Specifically, upon defining $\phi_i^d(x) = [1,~x, ~x^2,\cdots, x^d]$, we plot the average cost per team on the y-axis against the polynomial degree $d$ on the x-axis.
For this experiment, we fix the degree of $\phi_2(\cdot)$ to 2.
The values of $\lambda$ and $\alpha$ are set to 5 and 0.11, respectively. Figure \ref{Fig:CostVsDegree} shows that as the polynomial degree increases from linear to cubic, the average cost takes nearly similar values with their mean values recorded to be 77.2803, 79.7316, and 78.1966, respectively. This suggests that a quadratic degree for the basis functions provides a sufficiently accurate approximation of the value function within the assumed function class. While increasing the number of basis functions can enhance the representational power, it may also lead to overfitting, resulting in significantly larger parameter values and reduced generalization to new data points. Therefore, selecting the appropriate polynomial degree $d$ requires a careful balance between accuracy and generalization.

\begin{figure}[h]
	\centerline{\includegraphics[width=0.6\columnwidth]{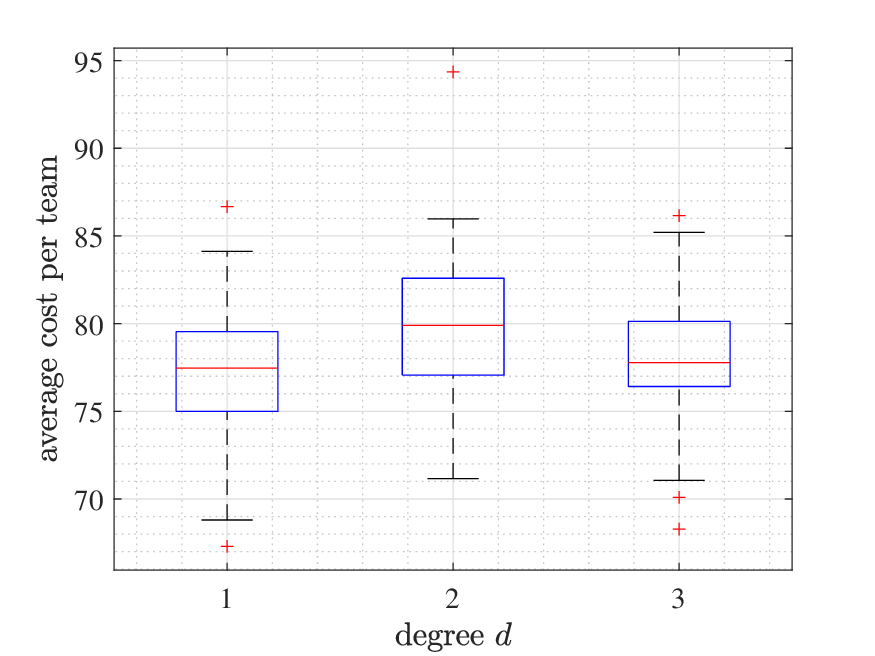}}
	\caption{\small{The plot shows the variation of the average cost per team versus the number of teams ($N$).}}
	\label{Fig:CostVsDegree}
    \vspace{-0.4cm} 
\end{figure}
\end{example}

\begin{example}
    In our second example, we simulate an OptiTrack motion capture system \cite{li2021real}, which serves as a representative prototype of our formulation. These systems typically involve multiple mobile robots, each analogous to a team in our formulation, that operate with onboard controllers. These controllers receive absolute state information via an array of infrared cameras mounted on the ceiling of an indoor facility. The motion capture system processes real-time positional data and transmits it over a shared Wi-Fi network, which is often bandwidth-constrained due to multiple devices communicating simultaneously. Such limitations can lead to network congestion, impacting the control performance. Thus, optimizing the transmission policy to prioritize essential state updates becomes crucial for maintaining precise control while mitigating the impact of network limitations.

    For our setup, we model the robot dynamics using second-order integrator dynamics, represented in continuous-time as:
    \begin{small}
    \begin{align*}
         dx^i(t) = \begin{bmatrix}
             0 & 1 \\
             0 & 0
         \end{bmatrix} dt + \begin{bmatrix}
             0 \\
             1
         \end{bmatrix} B^i u^i(t) dt + G^i dB^i(t), ~~\forall i \in [N],
    \end{align*}
    \end{small}
    \hspace{-1.8mm}where $B^i(t)$ denotes the standard Brownian motion, $G^i$ is a matrix of suitable dimensions, and we slightly abuse the notation $x$ to denote the state of the continuous-time system as well. Discretizing the above system with a step-size of $h$, one can arrive at its discrete-time state-space representation given by  \eqref{System_dynamics} with system matrices $A = [1, h; 0, 1], B = [0,h]^\top$. Moreover, for the ensuing simulations, we take the cost coefficients to be $Q = Diag(2,1)$ and $R = 1$ with initial state covariance to be $\Sigma_0 = 0.1$ and system noise covariance to be $K_W = 0.3$. Finally, we learn the value function in the same way as earlier within the linear function class with basis functions with degree 2, and conduct the following studies.

    \textbf{1) Variation of signals over time:} In our first analysis, we examine a system with $N=10$ teams and a horizon of $T = 41$ and illustrate the temporal evolution of various signals for a representative team (team 1) in Figure \ref{Fig:Sim_robots}. We set the parameters as $\lambda = 4$ and $\alpha = 0.3$. The top-left subplot displays the variation of the team's position over time along with its corresponding estimate. The top-right subplot similarly illustrates the evolution of velocity and its estimate over time. In both of these subplots, we also mark the instants of communication between the sensor and the controller of team 1, highlighting when state updates are transmitted over the network. Finally, while the bottom left subplot presents the estimation error at the controller (demonstrating that it remains bounded throughout the simulation), the bottom right subplot shows how the control input varies over time.

\begin{figure}[h]
	\centerline{\includegraphics[width=0.7\columnwidth]{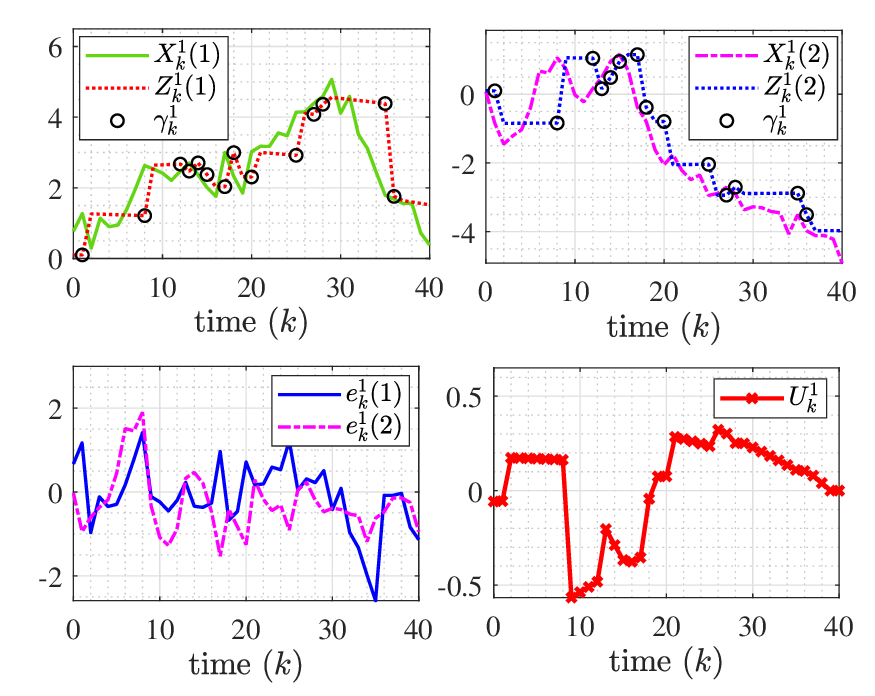}}
	\caption{\small{The plots show the variation of system states, state estimates, control input, communication instants and estimation errors over time for the team 1.}}
	\label{Fig:Sim_robots}
    \vspace{-0.4cm} 
\end{figure}

\textbf{2) Variation of number of communication instants versus $\mathbf{\lambda}$: }In our final study of this section, we analyze how the number of communication instants for a representative team (Team 1) varies with $\lambda$. To this end, we simulate a system with $N=10$ teams over a time horizon of $T=41$ and parameter $\alpha$ set to 0.07. Our results (from Table \ref{Table:comm_inst}) confirm that as $\lambda$ increases, the number of communication instants decreases, which aligns with intuition—higher communication costs encourage more sparse transmissions. Additionally, we examine the equilibrium channel utilization over time (in Fig. \ref{Fig:MF_traj_rob}) and observe a similar decreasing trend, demonstrating that increased communication costs lead to reduced channel occupancy. This finally concludes our numerical evaluation of our proposed mean-field team game approach.

\begin{table}[h]
\centering
\begin{tabular}{|c|c|c|}
\hline
$\lambda$ & \# comm. instants & \% comm. instants \\ \hline
0.3       & 26                & 65\%              \\ \hline
3         & 20                & 50\%              \\ \hline
14        & 15                & 37.5\%            \\ \hline
\end{tabular}
\caption{\small{Variation of number and percentage of communication instants for team 1 with cost coefficient $\lambda$ (for $T = 41$).}}
\end{table}
\label{Table:comm_inst}

\begin{figure}[h]
\centerline{\includegraphics[width=0.7\columnwidth]{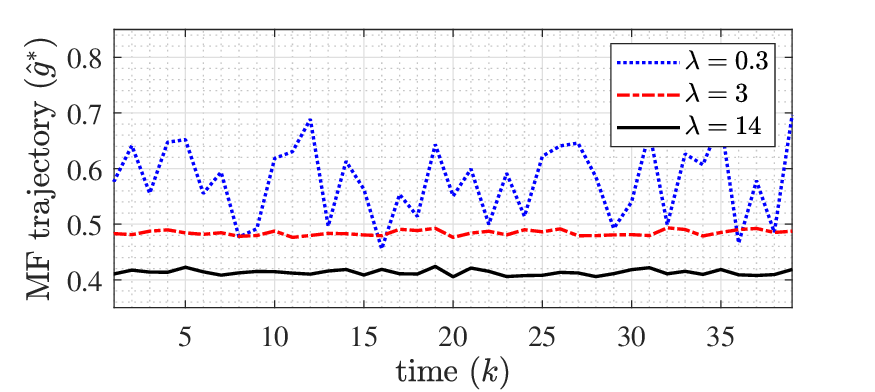}}
	\caption{\small{The plot shows the variation of the equilibrium MF trajectory over time.}}
	\label{Fig:MF_traj_rob}
    \vspace{-0.9cm} 
\end{figure}

\end{example}

\section{Conclusion \& Discussions}\label{sec:conclusion}
In this paper, we have solved for approximate Nash equilibrium team policies for a multi-team dynamic game problem where each team is coordinating its communication and control policies while utilizing a shared wireless channel. Specifically, we have formulated a noncooperative (general sum) game between $N$ teams, each of which was comprised of a local sensor and a local controller which communicate over a network and incur an overall cost which is coupled with the (sensor) policies of other teams as a result of communication over the shared channel. Due to unavailability of the sensor policies of other teams (which become challenging to obtain especially when there is a high population of them), we have employed the mean-field team game framework to compute \textit{approximate decentralized} Nash-equilibria between the teams. Toward this end, we have first noticed that the finite cardinality of the sensor action space prohibited the use of a contraction argument to establish the existence of a mean-field equilibrium. To overcome this difficulty, we have constructed a Boltzmann-type sensing policy which we have then employed to explicitly prove the existence of a unique approximate mean-field team equilibrium. Subsequently, we showed that the same constituted an approximate Nash equilibrium for the finite-team game, where the approximation depended vanishingly on the number of teams, $N$, and the Boltzmann parameter $\alpha$. We have also validated the theoretical results with extensive numerical simulations.

We list here a few future research directions. First, for the numerical analysis, we had assumed that the price parameter $\lambda$ was fixed and provided an extensive numerical analysis for its effect on the average costs, estimation errors, and control inputs of the teams. Thus one direction would be to formulate the finite-team game within a Stackelberg game setting, where the leader (possibly a telecommunication/network operator) solves an optimization problem to distribute `optimal' $\lambda$ to the follower teams. This necessitates substituting the policies of the followers into the leader's optimization problem, and subsequently deriving the leader's optimal policy. However, this (at the moment) seems to be quite challenging since each team's policy is not attainable in closed form; rather, we only know it as a function of the conditional value function. Thus, it would be interesting to see if one can derive analytical results for the Stackelberg game setting. 

Another direction that deserves mention here is that in the current work we did not account for the \textit{no-transmission} instants within the information structure of the controller. This was due to the fact that the no-transmission instants make the estimator dynamics nonlinear, and couples the design on the optimal estimator and the optimal sensing policy. Although it has been shown in \cite{soleymani2022value} that the optimal estimator turns out to be linear and decoupled from the sensor design, this is true only at optimality. However, upon careful observation, one may notice that in the current work, we have employed a Boltzmann sensing policy (which is an approximation to the optimal policy) to establish the existence of a unique and tractable equilibrium. Hence, the decoupling argument between the estimator and the scheduler no longer holds. Thus, it would be interesting to consider whether similar results as in this paper can be generalized to the case where the information set of the controller also includes the no-transmission instants. 

Finally, in the current work, we did not explicitly account for network-induced effects such as packet dropouts, channel fading, and stochastic delays, among others. Investigating how these factors would impact team performance and equilibrium strategies would be an interesting extension of our study.

\bibliographystyle{IEEEtran}
\bibliography{refs}

\begin{thebibliography}{10}
\providecommand{\url}[1]{#1}
\csname url@samestyle\endcsname
\providecommand{\newblock}{\relax}
\providecommand{\bibinfo}[2]{#2}
\providecommand{\BIBentrySTDinterwordspacing}{\spaceskip=0pt\relax}
\providecommand{\BIBentryALTinterwordstretchfactor}{4}
\providecommand{\BIBentryALTinterwordspacing}{\spaceskip=\fontdimen2\font plus
\BIBentryALTinterwordstretchfactor\fontdimen3\font minus
  \fontdimen4\font\relax}
\providecommand{\BIBforeignlanguage}[2]{{%
\expandafter\ifx\csname l@#1\endcsname\relax
\typeout{** WARNING: IEEEtran.bst: No hyphenation pattern has been}%
\typeout{** loaded for the language `#1'. Using the pattern for}%
\typeout{** the default language instead.}%
\else
\language=\csname l@#1\endcsname
\fi
#2}}
\providecommand{\BIBdecl}{\relax}
\BIBdecl

\bibitem{shishika2020review}
D.~Shishika and V.~Kumar, ``A review of multi agent perimeter defense games,''
  in \emph{Decision and Game Theory for Security: 11th International
  Conference, GameSec 2020, College Park, MD, USA, October 28--30, 2020,
  Proceedings 11}.\hskip 1em plus 0.5em minus 0.4em\relax Springer, 2020, pp.
  472--485.

\bibitem{subramanian2023mean}
J.~Subramanian, A.~Kumar, and A.~Mahajan, ``Mean-field games among teams,''
  \emph{arXiv preprint arXiv:2310.12282}, 2023.

\bibitem{weintraub2008markov}
G.~Y. Weintraub, C.~L. Benkard, and B.~Van~Roy, ``Markov perfect industry
  dynamics with many firms,'' \emph{Econometrica}, vol.~76, no.~6, pp.
  1375--1411, 2008.

\bibitem{hatanaka2015passivity}
T.~Hatanaka, N.~Chopra, M.~Fujita, and M.~W. Spong, \emph{Passivity-based
  {C}ontrol and {E}stimation in {N}etworked {R}obotics}.\hskip 1em plus 0.5em
  minus 0.4em\relax Springer, 2015.

\bibitem{marschak1955elements}
J.~Marschak, ``Elements for a theory of teams,'' \emph{Management Science},
  vol.~1, no.~2, pp. 127--137, 1955.

\bibitem{radner1962team}
R.~Radner, ``Team decision problems,'' \emph{The Annals of Mathematical
  Statistics}, vol.~33, no.~3, pp. 857--881, 1962.

\bibitem{yuksel2013stochastic}
S.~Y{\"u}ksel and T.~Ba{\c{s}}ar, \emph{Stochastic Networked Control Systems:
  Stabilization and Optimization under Information Constraints}.\hskip 1em plus
  0.5em minus 0.4em\relax Springer Science \& Business Media, 2013.

\bibitem{dave2019decentralized}
A.~Dave and A.~A. Malikopoulos, ``Decentralized stochastic control in partially
  nested information structures,'' \emph{IFAC-PapersOnLine}, vol.~52, no.~20,
  pp. 97--102, 2019.

\bibitem{dave2022decentralized}
A.~Dave, N.~Venkatesh, and A.~A. Malikopoulos, ``Decentralized control of two
  agents with nested accessible information,'' in \emph{2022 American Control
  Conference (ACC)}.\hskip 1em plus 0.5em minus 0.4em\relax IEEE, 2022, pp.
  3423--3430.

\bibitem{witsenhausen1968counterexample}
H.~S. Witsenhausen, ``A counterexample in stochastic optimum control,''
  \emph{{SIAM} Journal on Control}, vol.~6, no.~1, pp. 131--147, 1968.

\bibitem{feldbaum1961dual}
A.~A. Feldbaum, ``Dual control theory,'' in \emph{Control Theory: Twenty-Five
  Seminal Papers (1932-1981)}, T.~Ba\c{s}ar, Ed.\hskip 1em plus 0.5em minus
  0.4em\relax Wiley-IEEE Press, 2001, ch.~10, pp. 874--880.

\bibitem{basar2008variations}
T.~Ba\c{s}ar, ``Variations on the theme of the {W}itsenhausen counterexample,''
  in \emph{Proceedings of the 47th IEEE Conference on Decision and Control},
  December 2008, pp. 1614--1619.

\bibitem{tabuada2007event}
P.~Tabuada, ``Event-triggered real-time scheduling of stabilizing control
  tasks,'' \emph{IEEE Transactions on Automatic Control}, vol.~52, no.~9, pp.
  1680--1685, 2007.

\bibitem{heemels2012introduction}
W.~P. Heemels, K.~H. Johansson, and P.~Tabuada, ``An introduction to
  event-triggered and self-triggered control,'' in \emph{Proceedings of the
  51st IEEE Conference on Decision and Control}, 2012, pp. 3270--3285.

\bibitem{dimarogonas2011distributed}
D.~V. Dimarogonas, E.~Frazzoli, and K.~H. Johansson, ``Distributed
  event-triggered control for multi-agent systems,'' \emph{IEEE Transactions on
  Automatic Control}, vol.~57, no.~5, pp. 1291--1297, 2011.

\bibitem{seyboth2013event}
G.~S. Seyboth, D.~V. Dimarogonas, and K.~H. Johansson, ``Event-based
  broadcasting for multi-agent average consensus,'' \emph{Automatica}, vol.~49,
  no.~1, pp. 245--252, 2013.

\bibitem{imer2005optimal}
O.~C. Imer and T.~Ba{\c s}ar, ``Optimal estimation with limited measurements,''
  in \emph{Proceedings of the 44th IEEE Conference on Decision and
  Control}.\hskip 1em plus 0.5em minus 0.4em\relax IEEE, 2005, pp. 1029--1034.

\bibitem{imer2010optimal}
O.~C. Imer and T.~Ba{\c{s}}ar, ``Optimal estimation with limited
  measurements,'' \emph{International Journal of Systems, Control and
  Communications}, vol.~2, no. 1-3, pp. 5--29, 2010.

\bibitem{lipsa2011remote}
G.~M. Lipsa and N.~C. Martins, ``Remote state estimation with communication
  costs for first-order {LTI} systems,'' \emph{IEEE Transactions on Automatic
  Control}, vol.~56, no.~9, pp. 2013--2025, 2011.

\bibitem{imer2006optimal}
O.~C. Imer and T.~Ba{\c s}ar, ``Optimal control with limited controls,'' in
  \emph{2006 American Control Conference}.\hskip 1em plus 0.5em minus
  0.4em\relax IEEE, 2006, pp. 298--303.

\bibitem{imer2006measure}
O.~C. Imer and T.~Ba\c{s}ar, ``To measure or to control: optimal control with
  scheduled measurements and controls,'' in \emph{2006 American Control
  Conference}.\hskip 1em plus 0.5em minus 0.4em\relax IEEE, 2006, pp.
  1003--1008.

\bibitem{molin2009lqg}
A.~Molin and S.~Hirche, ``On {LQG} joint optimal scheduling and control under
  communication constraints,'' in \emph{Proceedings of the 48th IEEE Conference
  on Decision and Control held jointly with 28th Chinese Control
  Conference}.\hskip 1em plus 0.5em minus 0.4em\relax IEEE, 2009, pp.
  5832--5838.

\bibitem{maity2020minimal}
D.~Maity and J.~S. Baras, ``Minimal feedback optimal control of
  linear-quadratic-{G}aussian systems: No communication is also a
  communication,'' \emph{IFAC-PapersOnLine}, vol.~53, no.~2, pp. 2201--2207,
  2020.

\bibitem{soleymani2021valu}
T.~Soleymani, J.~S. Baras, and S.~Hirche, ``Value of information in feedback
  control: Quantification,'' \emph{IEEE Transactions on Automatic Control},
  vol.~67, no.~7, pp. 3730--3737, 2022.

\bibitem{soleymani2022value}
T.~Soleymani, J.~S. Baras, S.~Hirche, and K.~H. Johansson, ``Value of
  information in feedback control: Global optimality,'' \emph{IEEE Transactions
  on Automatic Control}, vol.~68, no.~6, pp. 3641--3647, 2023.

\bibitem{hogeboom2023zero}
I.~Hogeboom-Burr and S.~Y{\"u}ksel, ``Zero-sum games involving teams against
  teams: Existence of equilibria, and comparison and regularity in
  information,'' \emph{Systems \& Control Letters}, vol. 172, p. 105454, 2023.

\bibitem{lagoudakis2002learning}
M.~G. Lagoudakis and R.~Parr, ``Learning in zero-sum team {M}arkov games using
  factored value functions,'' \emph{Advances in Neural Information Processing
  Systems}, vol.~15, 2002.

\bibitem{ghimire2023solving}
M.~Ghimire, L.~Zhang, W.~Zhang, Y.~Ren, and Z.~Xu, ``Solving two-player
  general-sum games between swarms,'' \emph{available on arXiv:2310.01682},
  2023.

\bibitem{maity2017linear}
D.~Maity and J.~S. Baras, ``Linear quadratic stochastic differential games
  under asymmetric value of information,'' \emph{IFAC-PapersOnLine}, vol.~50,
  no.~1, pp. 8957--8962, 2017.

\bibitem{huang2007large}
M.~Huang, P.~E. Caines, and R.~P. Malham{\'e}, ``Large-population cost-coupled
  {LQG} problems with nonuniform agents: individual-mass behavior and
  decentralized $\varepsilon$-{N}ash equilibria,'' \emph{IEEE Transactions on
  Automatic Control}, vol.~52, no.~9, pp. 1560--1571, 2007.

\bibitem{lasry2007mean}
J.-M. Lasry and P.-L. Lions, ``Mean field games,'' \emph{Japanese Journal of
  Mathematics}, vol.~2, no.~1, pp. 229--260, 2007.

\bibitem{huang2006large}
M.~Huang, R.~P. Malham{\'e}, and P.~E. Caines, ``Large population stochastic
  dynamic games: closed-loop {M}c{K}ean-{V}lasov systems and the {N}ash
  {C}ertainty {E}quivalence principle,'' \emph{Communications in Information \&
  Systems}, vol.~6, no.~3, pp. 221--252, 2006.

\bibitem{aggarwal2023weighted}
S.~Aggarwal, M.~A.~U. Zaman, M.~Bastopcu, and T.~Ba{\c{s}}ar, ``Weighted age of
  information based scheduling for large population games on networks,''
  \emph{IEEE Journal on Selected Areas in Information Theory}, vol.~4, pp.
  682--697, 2023.

\bibitem{bagagiolo2014mean}
F.~Bagagiolo and D.~Bauso, ``Mean-field games and dynamic demand management in
  power grids,'' \emph{Dynamic Games and Applications}, vol.~4, pp. 155--176,
  2014.

\bibitem{olmez2022modeling}
S.~Y. Olmez, S.~Aggarwal, J.~W. Kim, E.~Miehling, T.~Ba{\c{s}}ar, M.~West, and
  P.~G. Mehta, ``Modeling presymptomatic spread in epidemics via mean-field
  games,'' in \emph{2022 American Control Conference (ACC)}.\hskip 1em plus
  0.5em minus 0.4em\relax IEEE, 2022, pp. 3648--3655.

\bibitem{aggarwal2024mean}
S.~Aggarwal, M.~A. uz~Zaman, M.~Bastopcu, S.~Ulukus, and T.~Ba{\c{s}}ar, ``A
  mean field game model for timely computation in edge computing systems,''
  \emph{available on arXiv:2404.02898}, 2024.

\bibitem{huang2024linear}
J.~Huang, Z.~Qiu, S.~Wang, and Z.~Wu, ``Linear quadratic mean-field game-team
  analysis: A mixed coalition approach,'' \emph{Automatica}, vol. 159, p.
  111358, 2024.

\bibitem{djehiche2016mean}
B.~Djehiche, A.~Tcheukam, and H.~Tembine, ``Mean-field-type games in
  engineering,'' \emph{available on arXiv:1605.03281}, 2016.

\bibitem{zaman2024independent}
M.~A.~U. Zaman, A.~Koppel, M.~Lauri{\`e}re, and T.~Ba{\c{s}}ar, ``Independent
  {RL} for cooperative-competitive agents: A mean-field perspective,''
  \emph{available on arXiv:2403.11345}, 2024.

\bibitem{aggarwal2023large}
S.~Aggarwal, M.~A. uz~Zaman, M.~Bastopcu, and T.~Ba{\c{s}}ar, ``Large
  population games on constrained unreliable networks,'' in \emph{2023 62nd
  IEEE Conf. on Decision and Control (CDC)}, 2023, pp. 3480--3485.

\bibitem{mason2024multi}
F.~Mason, F.~Chiariotti, A.~Zanella, and P.~Popovski, ``Multi-agent
  reinforcement learning for coordinating communication and control,''
  \emph{IEEE Transactions on Cognitive Communications and Networking}, 2024.

\bibitem{lan2021semantic}
Q.~Lan, D.~Wen, Z.~Zhang, Q.~Zeng, X.~Chen, P.~Popovski, and K.~Huang, ``What
  is semantic communication? {A} view on conveying meaning in the era of
  machine intelligence,'' \emph{Journal of Communications and Information
  Networks}, vol.~6, no.~4, pp. 336--371, 2021.

\bibitem{uysal2022semantic}
E.~Uysal, O.~Kaya, A.~Ephremides, J.~Gross, M.~Codreanu, P.~Popovski,
  M.~Assaad, G.~Liva, A.~Munari, B.~Soret \emph{et~al.}, ``Semantic
  communications in networked systems: A data significance perspective,''
  \emph{IEEE Network}, vol.~36, no.~4, pp. 233--240, 2022.

\bibitem{du2023rethinking}
H.~Du, J.~Wang, D.~Niyato, J.~Kang, Z.~Xiong, M.~Guizani, and D.~I. Kim,
  ``Rethinking wireless communication security in semantic internet of
  things,'' \emph{IEEE Wireless Comm.}, vol.~30, no.~3, pp. 36--43, 2023.

\bibitem{kountouris2021semantics}
M.~Kountouris and N.~Pappas, ``Semantics-empowered communication for networked
  intelligent systems,'' \emph{IEEE Communications Magazine}, vol.~59, no.~6,
  pp. 96--102, 2021.

\bibitem{ayan2022semantics}
O.~Ayan, P.~Kutsevol, H.~Y. {\"O}zkan, and W.~Kellerer, ``Semantics-and
  task-oriented scheduling for networked control systems in practice,''
  \emph{IEEE Access}, vol.~10, pp. 115\,673--115\,690, 2022.

\bibitem{yang2022semantic}
W.~Yang, H.~Du, Z.~Q. Liew, W.~Y.~B. Lim, Z.~Xiong, D.~Niyato, X.~Chi, X.~Shen,
  and C.~Miao, ``Semantic communications for future {I}nternet: Fundamentals,
  applications, and challenges,'' \emph{IEEE Communications Surveys \&
  Tutorials}, vol.~25, no.~1, pp. 213--250, 2022.

\bibitem{lasry2006jeux}
J.-M. Lasry and P.-L. Lions, ``Jeux {\`a} champ moyen. i--le cas
  stationnaire,'' \emph{Comptes Rendus Math{\'e}matique}, vol. 343, no.~9, pp.
  619--625, 2006.

\bibitem{lasry2006jeux1}
------, ``Jeux {\`a} champ moyen. ii--horizon fini et contr{\^o}le optimal,''
  \emph{Comptes Rendus Math{\'e}matique}, vol. 343, no.~10, pp. 679--684, 2006.

\bibitem{li2021real}
J.~Li, X.~Liu, Z.~Wang, H.~Zhao, T.~Zhang, S.~Qiu, X.~Zhou, H.~Cai, R.~Ni, and
  A.~Cangelosi, ``Real-time human motion capture based on wearable inertial
  sensor networks,'' \emph{IEEE Internet of Things Journal}, vol.~9, no.~11,
  pp. 8953--8966, 2021.

\bibitem{cui2021approximately}
K.~Cui and H.~Koeppl, ``Approximately solving mean field games via
  entropy-regularized deep reinforcement learning,'' in \emph{Intl. Conf. on
  Artificial Intelligence and Statistics}.\hskip 1em plus 0.5em minus
  0.4em\relax PMLR, 2021, pp. 1909--1917.

\bibitem{saldi2018markov}
N.~Saldi, T.~Ba{\c{s}}ar, and M.~Raginsky, ``Markov--{N}ash equilibria in
  mean-field games with discounted cost,'' \emph{SIAM Journal on Control and
  Optimization}, vol.~56, no.~6, pp. 4256--4287, 2018.

\bibitem{bertsekas1996neuro}
D.~Bertsekas and J.~N. Tsitsiklis, \emph{Neuro-Dynamic Programming}.\hskip 1em
  plus 0.5em minus 0.4em\relax Athena Scientific, 1996.

\end{thebibliography}

\section*{Appendix I}\label{App:existence}

\begin{proof}[Proof of Proposition \ref{prop:existence}]
    We prove the proposition by inductively establishing the continuity of the operator $\mathcal{T}$ and consequently employing Brouwer's fixed point theorem.

    To this end, let us define the following set
\begin{align} \label{eq:mathcal_G}
    \mathcal{G}:= \{\hat{g}:= \{\hat g_k\}_{k=0}^{T-1} \mid \|\hat{g}_k\| \leq 1,~\forall k \in [T]\}.
\end{align}
Then, we note from \eqref{MF_operator} that $\mathcal{T}: [0,1]^T \rightarrow [0,1]^T$, and hence, $\mathcal{T}(\mathcal{G}) \subseteq \mathcal{G}$. Further, the interval $[0,1]$ is compact and convex, and hence, $\mathcal{G}$ is also compact and convex. Thus, it only remains to prove the continuity of $\mathcal{T}$, which then by using Brouwer's fixed point theorem would prove the statement of the proposition. From now on, we will omit $\omega$ wherever it clutters the notation.

To establish continuity, using \eqref{MF_operator}, it suffices to show the continuity of $\operatorname{VoI}_k^{\omega}$ with respect to $\hat g$, for which we first show the continuity of the state-action value function, defined as:
\begin{align}
 \!\!{Q}^{*,x}_k(\calI^S_k,a_k) \!& = \EE[\lambda x_ka_k + e_{k+1}^\top \Gamma_k e_{k+1} + \min_{a \in \{0,1\}}{Q}^{*,x}_{k+1}(\calI^S_{k+1},a)],  
\end{align}
for all $k \in [T]$. The above denotes the value received by the generic team when the action $a_k$ is chosen at the current instant and the future actions are chosen optimally from the admissible space $\Xi^{dec}$ for a given MF trajectory $x$. We will use backward induction on $k$ to prove the result. Fix $\hat g \in \mathcal{G}$. Consider the base case when $k = T-1$. Then, by recalling that the terminal cost at instant $T$ is 0, we have that 
\begin{align}\label{test_Q1}
 \!\!{Q}^{*,\hat g}_{T-1}(\calI^S_{T-1},a_{T-1}) \!& = \EE[\lambda \hat g_{T-1} a_{T-1} + e_{T}^\top \Gamma_T e_{T}].
\end{align}

Upon using the dynamics \eqref{eq:est_error} in \eqref{test_Q1} for $k = T-1$, we arrive at the equation
\begin{align*}
  {Q}^{*,\hat g}_{T-1}&(\calI^S_{T-1},a_{T-1}) \!  = \lambda \hat g_{T-1} a_{T-1} + (1-a_{T-1})^2e_{T-1}^\top A^\top \Gamma_T Ae_{T-1} + \EE[W_{T-1}^\top \Gamma_T W_{T-1}],
\end{align*}
    which is linear (and thus, clearly continuous) in $\hat g$.

    Next, we take the induction hypothesis to be the assertion that ``${Q}^{*,\hat g}_{k+1}(\calI^S_{k+1},a_{k+1})$ is continuous in $\hat g$''. Then, we prove the assertion for instant $k$. Consider the following:
    \begin{small}
    \begin{align*}
 {Q}^{*,\hat g}_k(\calI^S_k,a_k) \!& = \EE[\lambda \hat g_ka_k + e_{k+1}^\top \Gamma_k e_{k+1} \! + \!\!\! \min_{a \in \{0,1\}}{Q}^{*,\hat g}_{k+1}(\calI^S_{k+1},a)].
\end{align*}
\end{small}
From the induction hypothesis, we have that $\min_{a \in \{0,1\}}{Q}^{*,\hat g}_{k+1}(\calI^S_{k+1},a)$ is continuous in $\hat{g}$ since the minimum of a continuous function is continuous. Further, the running cost of $\EE[\lambda \hat g_ka_k + e_{k+1}^\top \Gamma_k e_{k+1}]$ is also continuous in $\hat g$. The desired continuity of ${Q}^{*,\hat g}_k(\calI^S_k,a_k)$ thus follows from the fact that the sum of continuous functions is continuous. Consequently, the result is established for all $k \in [T]$ and any chosen action $a \in \{0,1\}$ by invoking the induction principle. The continuity of $\operatorname{VoI}_k^\omega$ then follows from by noting that 
\begin{align*}
    \operatorname{VoI}_k^\omega = {Q}^{*,\hat g,\omega}_k(\calI^S_k,0) - {Q}^{*,\hat g,\omega}_k(\calI^S_k,1)
\end{align*}
is the difference of two continuous functions. Finally, noting that $\mathcal{T}(x) \ne x, x = \{0\}^T,\{1\}^T$, the proof is complete.
\end{proof}

Next, to prove Theorem \ref{thm:uniqueness}, we first establish the following auxiliary lemma.

\begin{lemma}\label{lem:Lip_Q}
    Let $\hat g^1, \hat g^2 \in \mathcal{G}$ where $\mathcal{G}$ is defined in \eqref{eq:mathcal_G}. Then, the following is true for all $k \in [T]$:
    \begin{align*}
        {Q}^{*,\hat g^1}_{k}(\calI^S_{k},a_{k}) & - {Q}^{*,\hat g^2}_{k} (\calI^S_{k},a_{k}) \leq \lambda T \|\hat g^1 - \hat g^2\|.
    \end{align*}
\end{lemma}

\begin{proof}
    We will use backward induction on $k$ to prove the result. In this regard, first, consider the base case for $k = T-1$. Then, we have that:
    \begin{small}
\begin{align*}
    & {Q}^{*,\hat g^1}_{T-1}(\calI^S_{T-1},a_{T-1}) - {Q}^{*,\hat g^2}_{T-1}(\calI^S_{T-1},a_{T-1}) \\
 &  = \lambda \hat g^1_{T-1} a_{T-1} + (1-a_{T-1})^2e_{T-1}^\top A^\top \Gamma_T Ae_{T-1} + \EE[W_{T-1}^\top \Gamma_T W_{T-1}] - \lambda \hat g^2_{T-1} a_{T-1} \\
 & ~~~~ - (1-a_{T-1})^2e_{T-1}^\top A^\top \Gamma_T Ae_{T-1} - \EE[\tilde W_{T-1}^\top \Gamma_T \tilde W_{T-1}] \\
 & = \lambda a_{T-1} (\hat g^1_{T-1} - \hat g^2_{T-1}) \leq \lambda \|\hat g^1 - \hat g^2\|.
\end{align*}
\end{small}
The base case is thus true. Assume next that

\begin{align}\label{eq:ind_hyp}
    {Q}^{*,\hat g^1}_{k+1}(\calI^S_{k+1},a_{k+1}) & - {Q}^{*,\hat g^2}_{k+1}(\calI^S_{k+1},a_{k+1}) \leq \lambda (T-k-1) \|\hat g^1 - \hat g^2\|.
\end{align}
Then, consider the following:
\begin{align}\label{test_Q2}
    & {Q}^{*,\hat g^1}_{k}(\calI^S_{k},a_{k}) - {Q}^{*,\hat g^2}_{k} (\calI^S_{k},a_{k}) \nonumber \\
    & = \EE[\lambda \hat g^1_ka_k + e_{k+1}^\top \Gamma_k e_{k+1} + \min_{a \in \{0,1\}}{Q}^{*,\hat g^1}_{k+1}(\calI^S_{k+1},a)] \nonumber \\
    & - \EE[\lambda \hat g^2_ka_k + \tilde e_{k+1}^\top \Gamma_k \tilde e_{k+1} + \min_{a \in \{0,1\}}{Q}^{*,\hat g^2}_{k+1}(\calI^S_{k+1},a)] \nonumber \\
    &  = \lambda a_{k} (\hat g^1_{k} - \hat g^2_{k}) + \EE [ \min_{a \in \{0,1\}}{Q}^{*,\hat g^1}_{k+1} (\calI^S_{k+1},a) \! -\! \min_{a \in \{0,1\}}{Q}^{*,\hat g^2}_{k+1} (\calI^S_{k+1},a)],
\end{align}
where the errors $e_{k+1}$ and $\tilde{e}_{k+1}$ are under noise realizations $W_k$ and $\tilde W_k$, respectively.
Next, let us define $a_1^* := \text{argmin}_{a}~{Q}^{*,\hat g^1}_{k+1} (\calI^S_{k+1},a)$ and $a_2^* := \text{argmin}_{a}~{Q}^{*,\hat g^2}_{k+1} (\calI^S_{k+1},a)$. Then, we consider the following cases:

\textbf{1)} If $a_1^* = a_2^* = 1$ or $a_1^* = a_2^* = 0$, using \eqref{eq:ind_hyp}, we have that
\begin{small}
    \begin{align*}
        {Q}^{*,\hat g^1}_{k+1} (\calI^S_{k+1},a_1^*) & - {Q}^{*,\hat g^2}_{k+1} (\calI^S_{k+1},a_2^*) \leq \lambda (T-k-1) \|\hat g^1 - \hat g^2\|.
    \end{align*}
    \end{small}

    \textbf{2)} If $a_1^* = 1$ and $a_2^* = 0$, we have that
    \begin{small}
    \begin{align*}
        & {Q}^{*,\hat g^1}_{k+1} (\calI^S_{k+1},1) - {Q}^{*,\hat g^2}_{k+1} (\calI^S_{k+1},0) \\
        & \hspace{1.5cm} \leq {Q}^{*,\hat g^1}_{k+1} (\calI^S_{k+1},0) - {Q}^{*,\hat g^2}_{k+1} (\calI^S_{k+1},0) \\
        & \hspace{1.5cm} \leq \lambda (T-k-1) \|\hat g^1 - \hat g^2\|,
    \end{align*}
    \end{small}
    where the last inequality follows by \eqref{eq:ind_hyp}.

\textbf{3)} If $a_1^* = 0$ and $a_2^* = 1$, we have that
    \begin{align*}
        & {Q}^{*,\hat g^1}_{k+1} (\calI^S_{k+1},0) - {Q}^{*,\hat g^2}_{k+1} (\calI^S_{k+1},1) \\
        & \hspace{1.5cm} \leq {Q}^{*,\hat g^1}_{k+1} (\calI^S_{k+1},1) - {Q}^{*,\hat g^2}_{k+1} (\calI^S_{k+1},1) \\
        & \hspace{1.5cm} \leq \lambda (T-k-1) \|\hat g^1 - \hat g^2\|,
    \end{align*}
    where the last inequality follows by \eqref{eq:ind_hyp}.

Thus, continuing from \eqref{test_Q2}, we have that
\begin{align}
    {Q}^{*,\hat g^1}_{k}(\calI^S_{k},a_{k}) & - {Q}^{*,\hat g^2}_{k} (\calI^S_{k},a_{k}) \nonumber \\
    & \leq \lambda a_k (\hat g^1_k - \hat g^2_k) + \lambda (T-k-1) \|\hat g^1 - \hat g^2\| \nonumber \\
    & \leq \lambda (T-k) \|\hat g^1 - \hat g^2\| \leq \lambda T \|\hat g^1 - \hat g^2\|.
\end{align}
The proof of the lemma is thus completed by invoking the induction principle.
\end{proof}
We are now ready to prove Theorem \ref{thm:uniqueness}.
\begin{proof}[Proof of Theorem \ref{thm:uniqueness}]
    We will use Banach's fixed point theorem to prove the result. In this regard, we start by noting from \eqref{MF_operator} that $\mathcal{T}: [0,1]^T \rightarrow [0,1]^T$, and hence, $\mathcal{T}(\mathcal{G}) \subseteq \mathcal{G}$. Further, the interval $[0,1]$ is a closed subset of a complete metric space $\mathbb{R}$, and hence complete, which further implies that $[0,1]^T$ is complete under the sup norm. Finally, we prove that $\mathcal{T}$ is a contraction.

    Let $\hat g^1, \hat g^2 \in \mathcal{G}$ and further define $\operatorname{VoI}^{\omega,1}_k$ and $\operatorname{VoI}^{\omega,2}_k$ as the value of information terms at instant $k$ corresponding to $\hat g^1$ and $\hat g^2$, respectively. Then, consider the following:
    \begin{small}
\begin{align}\label{eq:cxtrction}
    & [\mathcal{T}(\hat{g}^1) - \mathcal{T}(\hat{g}^2)]_k  = \mathbb{E} \left[ \frac{1}{1+ e^{-\alpha \operatorname{VoI}^{\omega,1}_k}} \!-\! \frac{1}{1\!+\! e^{-\alpha \operatorname{VoI}^{\omega,2}_k}}\right] \nonumber \\
    & = \mathbb{E} \left[ \frac{e^{-\alpha \operatorname{VoI}^{\omega,2}_k} - e^{-\alpha \operatorname{VoI}^{\omega,1}_k}}{(1+ e^{-\alpha \operatorname{VoI}^{\omega,2}_k}) (1+ e^{-\alpha \operatorname{VoI}^{\omega,1}_k})} \right] \nonumber \\
    & = \mathbb{E} \left[ \frac{e^{-\alpha \operatorname{VoI}^{\omega,2}_k}(1 - e^{-\alpha \operatorname{VoI}^{\omega,1}_k + \alpha \operatorname{VoI}^{\omega,2}_k })}{(1+ e^{-\alpha \operatorname{VoI}^{\omega,2}_k}) (1+ e^{-\alpha \operatorname{VoI}^{\omega,1}_k})} \right] \nonumber \\
    & \leq  \mathbb{E} \!\left[ \frac{\alpha e^{-\alpha \operatorname{VoI}^{\omega,2}_k}}{(1\!+ \!e^{-\alpha \operatorname{VoI}^{\omega,2}_k}) (1\!+ \!e^{-\alpha \operatorname{VoI}^{\omega,1}_k})} (\operatorname{VoI}^{\omega,1}_k \!-\! \operatorname{VoI}^{\omega,2}_k) \right]. 
\end{align}
\end{small}
Next, we prove that $\operatorname{VoI}^{\omega}$ is Lipschitz continuous in $\hat g$. To this end, we invoke the definition of the value of information to arrive at:
\begin{align}\label{test_Q3}
    \operatorname{VoI}^{\omega,1}_k - \operatorname{VoI}^{\omega,2}_k & = {Q}^{*,\hat g^1,\omega}_k(\calI^{S,\omega}_k,0) - {Q}^{*,\hat g^1,\omega}_k(\calI^{S,\omega}_k,1) \nonumber \\
    & - {Q}^{*,\hat g^2,\omega}_k(\calI^{S,\omega}_k,0) +  {Q}^{*,\hat g^2,\omega}_k(\calI^{S,\omega}_k,1) \nonumber \\
    & = {Q}^{*,\hat g^1,\omega}_k(\calI^{S,\omega}_k,0) -  {Q}^{*,\hat g^2,\omega}_k(\calI^{S,\omega}_k,0) \nonumber \\
    & + {Q}^{*,\hat g^2,\omega}_k(\calI^{S,\omega}_k,1) -  {Q}^{*,\hat g^2,\omega}_k(\calI^{S,\omega}_k,1) \nonumber \\
    & \leq 2 \lambda (T-k) \|\hat g^1 - \hat g^2\|,
\end{align}
where the last inequality follows by using Lemma \ref{lem:Lip_Q}. Then, by substituting \eqref{test_Q3} in \eqref{eq:cxtrction}, we arrive at $    [\mathcal{T}(\hat{g}^1) - \mathcal{T}(\hat{g}^2)]_k \leq L_k\|\hat g^1 - \hat g^2\|$ where
\begin{align*}
    L_k:= 2 \alpha \lambda (T-k) \mathbb{E} \left[ \frac{e^{-\alpha \operatorname{VoI}^{\omega,2}_k}}{(1+ e^{-\alpha \operatorname{VoI}^{\omega,2}_k}) (1+ e^{-\alpha \operatorname{VoI}^{\omega,1}_k})} \right].
\end{align*}
 A similar inequality can be proven for $[\mathcal{T}(\hat{g}^2) - \mathcal{T}(\hat{g}^1)]_k$, which would then imply the above inequality for $|[\mathcal{T}(\hat{g}^1) - \mathcal{T}(\hat{g}^2)]|_k$, as well. Thus, the operator $\mathcal{T}$ is a contraction if $\max_{k \in [T]}L_k <1$, which is implied by the sufficient condition \eqref{contraction_condition}. The proof is then completed by invoking Banach's fixed point theorem to conclude that there exists a unique $\hat g^*$ which is the fixed point of the operator $\mathcal{T}$ such that $\hat g^* = \mathcal{T}(\hat g^*)$.
\end{proof}

\section*{Appendix II}
\label{App:EPS_NAsh}
\begin{lemma}\label{eq:First_term_bound}
    It holds that
    \begin{align*}
        J^N_i(\pi^{*,i}, \hat{\zeta}^{*,i},\hat{\zeta}^{*,-i}) - J^\omega(\pi^{*},\hat{\zeta},\hat{g}^*) \le \mathcal{O}\Big(\frac{1}{\sqrt{N}} + \epsilon_{P,N} \Big).
    \end{align*}
\end{lemma}
\begin{proof}
Consider the following:
\begin{small}
\begin{align*}
    & J^N_i(\pi^{*,i},\hat{\zeta}^{*,i},\hat{\zeta}^{*,-i}) \nonumber \\
    & = \! \frac{1}{T} \mathbb{E}\!\left[ \!\sum_{k=0}^{T-1} \! \|\hat{X}^i_k\|^2_{Q(\omega_i)} \!+\! \|U^i_k\|^2_{R(\omega_i)} \!+\! \lambda \hat{\gamma}^{N,av}_k \hat{\gamma}^i_k \! - \! \lambda \hat g^*_k \hat{\gamma}^i_k \!+\! \lambda \hat{g}^{*}_k \hat{\gamma}^i_k\right] \\
    & = J^\omega(\pi^{*},\hat{\zeta}^*,\hat{g}^{*}) + \frac{\lambda}{T}\mathbb{E}\left[  \sum_{k=0}^{T-1}(\hat{\gamma}^{N,av}_k - \hat{g}^{*}_k)\hat{\gamma}^i_k \right] \\
    & \leq J^\omega(\pi^{*},\hat{\zeta},\hat{g}^{*}) + \mathcal{O}\Big(\frac{1}{\sqrt{N}} + \epsilon_{P,N} \Big),
\end{align*}
\end{small}
\hspace{-1.8mm}where the equality follows since we use the approximate MF policy pair ($\pi^*,\hat{\zeta}^*$) on the finite-agent system and the inequality follows as a result of Lemma \ref{lem:Approx_MFE}. The proof is thus complete.
\end{proof}

\begin{lemma}\label{eq:Third_term_bound}
    The following inequality holds:
    \begin{align}\label{Bound3}
        & \inf_{\zeta^d \in \Xi^{dec}} J^\omega(\pi^{*},{\zeta^d},\hat{g}^*) - \inf_{\xi^{i} \in \Xi^{cen,i}} J_i^N(\pi^{*,i},\xi^{i}, \hat{\zeta}^{*,-i}) \le \mathcal{O}\Big( \frac{1}{\sqrt{N}} + \epsilon_{P,N}\Big).
    \end{align}
\end{lemma}
    \begin{proof}
        Let us fix a policy $\xi^{i} \in \Xi^{cen,i}$ and $\tilde \gamma^i_k \sim \xi^i(\calI^{S_i}_k)$ while the other agents play using the policy $\hat \zeta^*$. Then, we have that:
\begin{align*}
    & J_i^N(\pi^{*,i},\xi^i, \hat{\zeta}^{*,-i}) \\
    & = \frac{1}{T} \mathbb{E}\left[ \sum_{k=0}^{T-1} \|{X}^i_k\|^2_{Q(\omega_i)} + \|U^i_k\|^2_{R(\omega_i)} + \lambda \tilde \gamma^{N,av}_k \tilde \gamma^i_k +\lambda \hat g^*_k\tilde \gamma^i_k - \lambda \hat g^*_k\tilde \gamma^i_k\right] \\
    & = \frac{1}{T} \mathbb{E}\left[ \sum_{k=0}^{T-1} \|{X}^i_k\|^2_{Q(\omega_i)} + \|U^i_k\|^2_{R(\omega_i)} + \lambda \hat g^*_k \gamma^i_k + \lambda  \tilde \gamma^{N,av}_k \tilde \gamma^i_k - \lambda \hat g^*_k\tilde \gamma^i_k\right] \\
    & \geq \inf_{\zeta^d \in \Xi^{dec}} J^\omega(\pi^{*},{\zeta^d},\hat g^*) + \frac{1}{T} \mathbb{E}\left[ \sum_{k=0}^{T-1}\lambda\tilde \gamma^i_k (  \tilde \gamma^{N,av}_k - \hat g^*_k) \right]
\end{align*}
which gives
\begin{align*}
    \inf_{\zeta^d \in \Xi^{dec}} & J^\omega(\pi^*,{\zeta}^d,\gamma^{av,*})  - J_i^N(\pi^{*,i},\xi^{i}, \hat{\zeta}^{*,-i}) \nonumber \\
    & \leq \frac{\lambda}{T}\mathbb{E}\left[  \sum_{k=0}^{T-1}(\hat g^*_k - \tilde \gamma^{N,av}_k)\tilde \gamma^i_k \right] \le \mathcal{O}\Big( \frac{1}{\sqrt{N}} + \epsilon_{P,N}\Big)
\end{align*}
which by taking infimum on both sides leads to \eqref{Bound3}.
The proof is thus complete.
    \end{proof}

\section*{Appendix III}
\label{App:eps:Term2}
\begin{lemma}\label{lem:Q_bounded}
The following inequality holds:
    \begin{align}
        \hat Q^{\hat g^*}_{k}(\calI^S_{k},\hat \gamma_{k}) \leq f_{k}(\|e_k\|^2) + c_k,~~~k \in [T],
    \end{align}
    where $\hat{Q}^x_k(\calI^S_k,a_k)$ denotes the state-action value function and is defined in a recursive manner as:
\begin{align}
    \!\!\hat{Q}^x_k(\calI^S_k,a_k) \!& = \!\EE[\lambda x_ka_k \!+\! e_{k+1}^\top \Gamma_k e_{k+1} \!\!+\! \hat{Q}^x_{k+1}(\calI^S_{k+1},a_{k+1})],
\end{align}
with $a_k\sim \hat \zeta^*_k(\calI^S_k,x)$ and $a_{k+1} \sim \hat \zeta^*_{k+1}(\calI^S_{k+1},x)$.
Further, $f_k(\cdot)$ are time-dependent functions, and $c_k>0$ is a constant.
\end{lemma}
We note that the same inequality can also be established for $Q^{*,x}(\cdot, \cdot)$ using the same proof technique as presented below.

\begin{proof}
    We prove the above lemma by induction on $k$. In this regard, consider the base case with $k=T-1$. Then, we have
    \begin{align}
        & \hat Q^{\hat g^*}_{T-1}(\calI^S_{T-1},\hat \gamma_{T-1})  = \EE[\lambda \hat g^*_{T-1} \hat \gamma_{T-1} + e_T^\top \Gamma_T e_T] \nonumber \\
        & \leq \lambda \! + \!\!\sup_{0 \leq t \leq T}\|\Gamma_t\| \EE [e_T^\top e_T] \leq \lambda \!+\! \!\sup_{0 \leq t \leq T}\|\Gamma_t\| (\|e_{T-1}\|^2 \!+ \!K_W) \nonumber \\
        & =: f_{T-1}(\|e_{T-1}\|^2) + c_{T-1},
    \end{align}
    where we have used the error dynamics in \eqref{eq:est_error} to arrive at the last inequality. The base case is thus proven. Assume now that the inequality in the statement of the lemma holds for instant $k+1$, i.e., 
    \begin{align}\label{test_eqn_k}
        \hat Q^{\hat g^*}_{k+1}(\calI^S_{k+1},\hat \gamma_{k+1}) \leq f_{k+1}(\|e_{k+1}\|^2) + c_{k+1}.
    \end{align}

    Then, we prove \eqref{test_eqn_k} for instant $k$. Consider the following:
    \begin{small}
    \begin{align}
         & \hat Q^{\hat g^*}_{k}(\calI^S_{k},\hat \gamma_{k}) = \EE[\lambda \hat g^*_{k} \hat \gamma_{k} + e_{k+1}^\top \Gamma_{k+1} e_{k+1} + \hat{Q}^{\hat g^*}_{k+1}(\calI^S_{k+1},\hat \zeta^*)] \nonumber \\
         & \leq \lambda + \sup_{0 \leq t \leq T}\|\Gamma_t\| (\|e_{k}\|^2 + K_W)  + \EE[f_{k+1}(\|e_{k+1}\|^2) + c_{k+1}] \nonumber \\
         & \leq \lambda + \sup_{0 \leq t \leq T}\|\Gamma_t\| (\|e_{k}\|^2 + K_W) + \EE[ f_{k+1}(\|(1-\hat \gamma_k) e_{k} + W_k\|^2) + c_{k+1}] \nonumber \\
         & =: f_k(\|e_k\|^2) + c_k.
    \end{align}
\end{small}
The proof now follows by the induction principle, and is thus complete.    
\end{proof}

\begin{proposition}
    We have that
    \begin{align}
    \hat{Q}^{\hat g^*}_k(\calI^S_k,\hat \gamma_k) - {Q}^{*,\hat g^*}(\calI^S_k,\hat \gamma_k) \leq \epsilon'(\alpha).
\end{align}
 for all $k \in [T]$, where $\epsilon'(\alpha):= \mathcal{O} \Big( \frac{1}{1+ \min_{t \in [T_k]} \{e^{\alpha |\operatorname{VoI}_t|}\}}\Big)$ and $[T_k] \subseteq \{k, k+1, \cdots,t, \cdots, T-1\}$ is a subset of all those instants $t$ such that $\operatorname{VoI}_t \ne 0$.
\end{proposition}

\begin{proof}
Let us fix $\hat g^* \in [0,1]^T$. We will use induction on $k$ to prove the result.

 Let us start with the base case of $k=T-1$. Let us define the estimation errors
\begin{align*}
    \tilde e_{k+1} &= (1-\hat \gamma_k) A\hat e_k + \tilde W_k \\
    \bar e_{k+1} &= (1-\hat \gamma_k) A\hat e_k + \bar W_k,
\end{align*}
where $\hat e_k$ is the error at instant k under action $\hat \gamma_{k-1}$.
 
 We recall that $V_{T}(\mathcal{I}_T^S,\hat g^*) = 0$. Then, we have that
 \begin{align}
     & \hat{Q}^{\hat g^*}_{T-1}(\calI^S_{T-1},\hat \gamma_{T-1}) - {Q}^{*,\hat g^*}_{T-1}(\calI^S_{T-1},\hat \gamma_{T-1}) \nonumber \\
     & = tr(\Gamma_T(\EE[\tilde e_T^\top \tilde e_T] -\EE[\bar e_T^\top \bar {e}_T])) \nonumber \\
     & = tr(\Gamma_T ((1-\hat \gamma_{T-1} )A \hat e_{T-1} \hat e_{T-1}^\top A^\top) + K_W) - tr(\Gamma_T ((1-\hat \gamma_{T-1} )A \hat e_{T-1} \hat e_{T-1}^\top A^\top) + K_W) \nonumber \\
     & = 0 < \epsilon'(\alpha).
 \end{align}

The base case is thus proved. Assume now that the following is true:
\begin{align*}
    \hat{Q}^{\hat g^*}_{k+1}(\calI^S_{k+1},\hat \gamma_{k+1}) & - {Q}^{*,\hat g^*}_{k+1}(\calI^S_{k+1},\hat \gamma_{k+1}) \leq \epsilon_1(\alpha)/2 := \mathcal{O} \Big( \frac{1}{1+ \min_{t \in [T_{k+1}]} \{e^{\alpha |\operatorname{VoI}_t|}\}}\Big).
\end{align*}
Then, we prove the inequality in the statement of the Proposition for instant $k$. In this regard, let us define the set $\mathcal{A}:=\{0,1\}$ and the set $\mathcal{A}^{x}_{\text{opt}}(\calI^S)$ to be the set of optimal actions for a given trajectory $x$ and the information set $\calI^S$. Further let us denote a suboptimal action as $a_{\text{sub}} \in \mathcal{A} \setminus \mathcal{A}^s_{\text{opt}}(\calI^S)$.
Then, we have that
\begin{align*}
    & \hat{Q}^{\hat g^*}_{k}(\calI^S_{k},\hat \gamma_{k}) - {Q}^{*,\hat g^*}_{k}(\calI^S_{k},\hat \gamma_{k}) = \sum_{a \in \mathcal{A}} \hat \zeta^*_k(a \mid \calI^S_k, \hat g^{*,k})\hat Q^{\hat g^*}_{k+1}(\calI^S_{k+1},a) - \min_{a' \in \mathcal{A}} Q^{*,\hat g^*}_{k+1}(\calI^S_{k+1},a') \nonumber \\
    & = \hspace{-3mm} \sum_{a \in \mathcal{A}^{\hat g^*}_{\text{opt}}(\calI^S_k)} \hspace{-3mm}\hat \zeta^*_k(a \mid \calI^S_{k}, \hat g^{*,k})\hat Q^{\hat g^*}_{k+1}(\calI^S_{k+1},a) - \min_{a' \in \mathcal{A}} Q^{*,\hat{g}^*}_{k+1}(\calI^S_{k+1},a') + \hspace{-6mm}\sum_{\hspace{3mm} a \in \mathcal{A} \setminus \mathcal{A}^{\hat g^*}_{\text{opt}}(\calI^S_k)} \hspace{-4mm}\hat \zeta^*_k(a \mid \calI^S_{k}, \hat g^*)\hat Q^{\hat g^*}_{k+1}(\calI^S_{k+1},a) \nonumber \\
    & =  \Big[ \sum_{a \in \mathcal{A}^{\hat g^*}_{\text{opt}}(\calI^S_k)} \hat \zeta^*_k(a \mid \calI^S_k, \hat g^*)\hat Q^{\hat g^*}_{k+1}(\calI^S_{k+1},a) - \sum_{a \in \mathcal{A}^{\hat g^*}_{\text{opt}}(\calI^S_k)} \hat \zeta^*_k(a \mid \calI^S_k, \hat g^*) \min_{a' \in \mathcal{A}} Q^{*, \hat g^*}_{k+1}(\calI^S_{k+1},a') \Big ]  \nonumber \\
    & \hspace{1cm} +  \Big[ \sum_{a \in \mathcal{A}^{\hat g^*}_{\text{opt}}(\calI^S_k)} \hat \zeta^*_k(a \mid \calI^S_k, \hat g^*)  \min_{a' \in \mathcal{A}} Q^{*, \hat g^*}_{k+1}(\calI^S_{k+1},a') - \min_{a' \in \mathcal{A}} Q^{*, \hat g^*}_{k+1}(\calI^S_{k+1},a') \Big] \nonumber \\
    & \hspace{6mm} + \hspace{-3mm}\sum_{\hspace{3mm} a \in \mathcal{A} \setminus \mathcal{A}^{\hat g^*}_{\text{opt}}(\calI^S_k)} \hspace{-4mm}\hat \zeta^*_k(a \mid \calI^S_{k}, \hat g^*)\hat Q^{\hat g^*}_{k+1}(\calI^S_{k+1},a) \nonumber \\
    & \le 2 \max_{a \in \mathcal{A}^{\hat g^*}_{\text{opt}}(\calI^S_k)}[\hat Q^{\hat g^*}_{k+1}(\calI^S_{k+1},a) - \min_{a' \in \mathcal{A}} Q^{*, \hat g^*}_{k+1}(\calI^S_{k+1},a') ] +  \Big[ \sum_{a \in \mathcal{A}^{\hat g^*}_{\text{opt}}(\calI^S_k)} \hat \zeta^*_k(a \mid \calI^S_k, \hat g^*)  \min_{a' \in \mathcal{A}} Q^{*, \hat g^*}_{k+1}(\calI^S_{k+1},a') \nonumber \\
    & \hspace{1cm} - \min_{a' \in \mathcal{A}} Q^{*, \hat g^*}_{k+1}(\calI^S_{k+1},a') \Big] + \hspace{-3mm}\sum_{\hspace{3mm} a \in \mathcal{A} \setminus \mathcal{A}^{\hat g^*}_{\text{opt}}(\calI^S_k)} \hspace{-4mm}\hat \zeta^*_k(a \mid \calI^S_{k}, \hat g^*)\hat Q^{\hat g^*}_{k+1}(\calI^S_{k+1},a) \nonumber \\
    & \le \epsilon_1(\alpha) +  \Big[ \sum_{a \in \mathcal{A}^{\hat g^*}_{\text{opt}}(\calI^S_k)} \hat \zeta^*_k(a \mid \calI^S_k, \hat g^*)  \min_{a' \in \mathcal{A}} Q^{*, \hat g^*}_{k+1}(\calI^S_{k+1},a') - \min_{a' \in \mathcal{A}} Q^{*, \hat g^*}_{k+1}(\calI^S_{k+1},a') \Big] \nonumber \\
    & \hspace{6mm} + \hspace{-3mm}\sum_{\hspace{3mm} a \in \mathcal{A} \setminus \mathcal{A}^{\hat g^*}_{\text{opt}}(\calI^S_k)} \hspace{-4mm}\hat \zeta^*_k(a \mid \calI^S_{k}, \hat g^*)\hat Q^{\hat g^*}_{k+1}(\calI^S_{k+1},a) \nonumber \\
    & \markthis{(1)}{b:1}{\le} \epsilon_1(\alpha) + \hat \zeta^*_k(a_{\text{sub}} \mid \calI^S_{k}, \hat g^*) \min_{a' \in \mathcal{A}} Q^{*, \hat g^*}_{k+1}(\calI^S_{k+1},a') + \hspace{-3mm}\sum_{\hspace{3mm} a \in \mathcal{A} \setminus \mathcal{A}^{\hat g^*}_{\text{opt}}(\calI^S_k)} \hspace{-4mm}\hat \zeta^*_k(a \mid \calI^S_{k}, \hat g^*)\hat Q^{\hat g^*}_{k+1}(\calI^S_{k+1},a) \nonumber\\
    & \le \epsilon_1(\alpha) + 4(f_k(\|e_k\|^2) + c_k) \frac{1}{1+e^{\alpha |\operatorname{VoI}_{k}|}}
    \leq \epsilon'(\alpha).
\end{align*}

We note that the above is valid for $\operatorname{VoI} \ne 0$. If it is 0, then we have that both the actions of 0 and 1 produce the same cost, and hence, there is no suboptimal action. Thus from \ref{b:1}, we get that 

$\hat{Q}^{\hat g^*}_{k}(\calI^S_{k},\hat \gamma_{k}) - {Q}^{*,\hat g^*}_{k}(\calI^S_{k},\hat \gamma_{k}) \le \epsilon_1(\alpha) < \epsilon'(\alpha)$.

The proof thus follows by the induction principle for all $ k \in [T]$, and is complete.
\end{proof}

\begin{proposition}\label{prop:Second_term_bound}
    It holds that
    \begin{align*}
        J^\omega(\pi^{*},\hat{\zeta}^*,\hat{g}^*) - \inf_{\zeta^d \in \Xi^{dec}} J^\omega(\pi^{*},{\zeta^d},\hat{g}^*) \le \epsilon(\alpha).
    \end{align*}
\end{proposition}

\begin{proof}

Let $\zeta^* = \text{arginf}_{\zeta^d \in \Xi^{dec}} J^\omega(\pi^*,\zeta^d, \hat g^*)$ and $\gamma^*$ be an action chosen from $\zeta^*$.
   \begin{align}
        & J^\omega(\pi^*,\hat{\zeta}^*,\hat g^*) - J^\omega(\pi^*,{\zeta}^*,\hat g^*)  \nonumber \\
        & = \mathbb E [\hat{Q}_0(e_0,\hat{\gamma}_0,\hat{g}^*)] - \mathbb E [{Q}^*_0(e_0,{\gamma}^*_0,\hat{g}^*)] \nonumber \\
        & = \mathbb E [\hat{Q}_0(e_0,\hat{\gamma}_0,\hat{g}^*)] - \mathbb E [{Q}^*_0(e_0,\hat{\gamma}_0,\hat{g}^*)]  \nonumber \\
        & \hspace{1cm} + \mathbb E [{Q}^*_0(e_0,\hat{\gamma}_0,\hat{g}^*)] - \mathbb E [{Q}^*_0(e_0,{\gamma}^*_0,\hat{g}^*)] \nonumber \\
        & \le \epsilon'(\alpha) + \mathbb E [{Q}^*_0(e_0,\hat{\gamma}_0,\hat{g}^*)] - \mathbb E [{Q}^*_0(e_0,{\gamma}^*_0,\hat{g}^*)] \nonumber \\
        & = \epsilon'(\alpha) + \mathbb E [\lambda \hat{g}^*_0 (\hat{\gamma}_0 - \gamma^*_0) + \hat{e}_1^\top\Gamma_1\hat{e}_1 - (e^*_1)^\top \Gamma_1 e^*_1 \nonumber \\
        & \hspace{1cm} + Q_1^*(\hat{e}_1,\gamma_1^*,\hat{g}^*) - Q_1^*( e^*_1,\gamma_1^*,\hat{g}^*)] \nonumber \\
        & = \epsilon'(\alpha) + \mathbb E [ \lambda \hat g^*_0 (\hat \gamma_0 - \gamma^*_0)] + tr(\Gamma_1 \mathbb E [( \gamma^*_0 - \hat \gamma_0)A  e_0 e_0^\top A^\top]) \nonumber \\
        & \hspace{1cm} + \mathbb E [Q_1^*((1-\hat \gamma_0)Ae_0 + \hat W_0,\gamma^*_1,\hat{g}^*) \nonumber \\
        & \hspace{1cm} - Q_1^*((1- \gamma^*_0)Ae_0 + \tilde W_0,\gamma^*_1,\hat{g}^*] \nonumber \\
        & \le \epsilon'(\alpha) + \mathbb E [ \lambda \hat g^*_0 (\hat \gamma_0 - \gamma^*_0)] + tr(\Gamma_1 Ae_0e_0^\top A^\top)\mathbb E [( \gamma^*_0 - \hat \gamma_0)] \nonumber \\
        & \hspace{1cm} + 2(f_0(\|e_0\|^2) + c_0) \mathbb P(\gamma^*_0 \neq \hat \gamma_0) \nonumber \\
        & \hspace{1cm} + \mathbb P( \gamma^*_0 = \hat \gamma_0) (\mathbb E [Q_1^*((1-\hat \gamma_0)Ae_0 + \hat W_0,\gamma^*_1,\hat{g}^*) \nonumber \\
        & \hspace{1cm} - Q_1^*((1-\hat \gamma_0)Ae_0 + \check W_0,\gamma^*_1,\hat{g}^*)] )\nonumber \\
         & \le \epsilon'(\alpha) + (\lambda + tr(\Gamma_1 Ae_0e_0^\top A^\top)) \mathbb E [( \gamma^*_0 - \hat \gamma_0)] \nonumber \\
         & \hspace{1cm} + 2(f_0(\|e_0\|^2) + c_0) \mathbb P( \gamma^*_0 \neq \hat \gamma_0) \nonumber \\ 
         & \leq \epsilon'(\alpha) + (\lambda + tr(\Gamma_1 Ae_0 e_0^\top A^\top)+ 2f_0(\|e_0\|^2) + 2c_0) \nonumber \\
         & \hspace{1cm} \times \Big(\frac{1}{1+ e^{\alpha |\operatorname{VoI}_{0,1}|}} + \frac{1}{1+ e^{\alpha |\operatorname{VoI}_{0,2}|}}\Big) \nonumber \\
         & =: \epsilon'(\alpha) + \epsilon''(\alpha) \nonumber \\
         & =: \epsilon(\alpha) = \mathcal{O} \Big( \frac{1}{1+ \min_{k \in [\bar T]} \{e^{\alpha |\operatorname{VoI}_k|}\}}\Big).
    \end{align}
The proof is thus complete.
\end{proof}

 \end{document}